\documentclass{amsart}
\usepackage{latexsym}
\usepackage{amstext}
\usepackage{amssymb}
\usepackage{amsfonts}
\usepackage{amsthm}
\usepackage{amsopn}
\usepackage{amsbsy}
\usepackage{layout}
\usepackage{color}
\usepackage{amsmath}
\usepackage{graphicx}
\usepackage{bm}
\usepackage{yfonts}
\usepackage{pifont}
\usepackage[scr]{rsfso}
\newtheorem{deff}{Definition}[section]
\newtheorem{lemma}[deff]{Lemma}
\newtheorem{theorem}[deff]{Theorem}%[subsection]
\newtheorem{corollary}[deff]{Corollary}%[subsection]
\newtheorem{conjecture}[deff]{Conjecture}
\newtheorem{proposition}[deff]{Proposition}
\newtheorem{fact}[deff]{Fact}
\newtheorem{em-example}[deff]{Example}
\newtheorem{em-def}[deff]{Definition}        % definition(auxiliary)
\newtheorem{em-remark}[deff]{Remark}         % remark(auxiliary)
\newtheorem{em-question}[deff]{Question}

\newtheorem{problem}[deff]{Problem}

\newenvironment{example}{\begin{em-example} \em }{ \end{em-example}}

\newenvironment{remark}{\begin{em-remark} \em }{\end{em-remark}}

\newenvironment{claim}{\underline{Claim}: \em}{\medskip}
\newcommand{\dom}{\text{dom}}

\newcommand{\acal}{{\mathcal A}}

\newcommand{\fcal}{\mathcal {F}}

\newcommand{\bDiamond}{\mathbin{\Diamond}}
\newcommand{\Lev}{\text{Lev}}
\newcommand{\ord}{\text{ord}}

\def\cl{\mathop{\it cl}}

\DeclareMathSymbol{\res}{\mathord}{AMSa}{"16}

\def\:{\nobreak \hskip .1111em\mathpunct {}\nonscript \mkern
   -\thinmuskip {:}\hskip .3333emplus.0555em\relax}

\catcode`\@=12

\def\N{{\mathbb N}}
\def\R{{\mathbb R}}
\def\S{{\mathbb S}}
\def\Q{{\mathbb Q}}

\def\cont{\mathfrak c}
\def\a{\mathfrak a}
\def\b{\mathfrak b}
\def\p{\mathfrak p}

\begin{document}
\title[On $\Delta$-spaces]{On $\Delta$-spaces}
\date{\today}

\author{Arkady Leiderman}
\address{Department of Mathematics, Ben-Gurion University of the Negev, Beer Sheva, Israel}
\email{arkady@math.bgu.ac.il}
\author{Paul Szeptycki}
\address{Department of Mathematics and Statistics, York University, Toronto, Canada}
\email{szeptyck@yorku.ca}

\begin{abstract}
$\Delta$-spaces have been defined by a natural generalization of a classical notion of $\Delta$-sets of reals to 
Tychonoff topological spaces; moreover, the class $\Delta$ of all $\Delta$-spaces consists precisely of those $X$ for which the locally convex space $C_p(X)$ is distinguished \cite{KL}.

The aim of this article is to better understand the boundaries of the class $\Delta$,
by presenting new examples and counter-examples.
 
1) We examine when trees considered as topological spaces equipped with the interval topology belong to $\Delta$.
In particular, we prove that no Souslin tree is a $\Delta$-space.
Other main results are connected with the study of
2) $\Psi$-spaces built on maximal almost disjoint families of countable sets; and 3) Ladder system spaces.

It is consistent with CH that all ladder system spaces on $\omega_1$ are in $\Delta$.
We show that in forcing extension of ZFC obtained by adding one Cohen real, there is a ladder system space on $\omega_1$ which is not in $\Delta$.

We resolve several open problems posed in \cite{FS}, \cite{KL}, \cite{Lei_Tkachenko}, \cite{Lei_Tkachuk}.
\end{abstract}
\maketitle
%%%%%%%%%%%%%%%%%%%%%%%%%%%%
\section{Introduction}\label{intro}
%%%%%%%%%%%%%%%%%%%%%%%%%%%
The reader is referred to consequent sections for an unexplained terminology.
Throughout the article, all topological spaces are assumed to be Tychonoff and infinite.
By $C_p(X)$ we mean the space $C(X)$ of real-valued continuous functions defined on $X$
equipped with the pointwise convergence topology. 
In this paper we investigate the class of topological spaces $X$
 such that the locally convex space $C_p(X)$ is distinguished. 
{\em Distinguished} locally convex (metrizable) spaces were introduced by J. Dieudonn\'{e} and L. Schwartz 
%as those locally convex spaces $E$ whose strong dual $E^{\ast}_{\beta}$ is barrelled.
and the reader may consult \cite{FKLS} (and references therein) for an introduction to and a brief history of the main advances 
in the area of general distinguished locally convex spaces.

%\begin{theorem}{\rm (\cite{FKLS}}\label{Theor:characterization}
%For a Tychonoff space $X$, the following conditions are equivalent{\rm:}
%\begin{enumerate}
%\item[{\rm (1)}] $C_p(X)$ is distinguished.
%\item[{\rm (2)}]  $C_{p}\left( X\right) $ is a large subspace of
%$\mathbb{R}^{X}$, i.e. for every mapping $f \in \R^X$ there is a bounded set $A \subset C_p(X)$ such that
%$f \in cl_{\R^X}(A)$.
%\item[{\rm (3)}] The strong dual $L_{\beta}(X)$ of the space $C_{p}(X)$ carries the finest locally convex topology.
%\end{enumerate}
%\end{theorem}

An intrinsic description of $X$ such that $C_p(X)$ is distinguished
 has been found recently in \cite{KL} (see also \cite{FS}).
% For the reader's convenience we recall some relevant terminology.
%\begin{enumerate}
%\item[\rm (a)] A collection of sets $\{U_{\gamma}: \gamma \in \Gamma\}$ is called an {\it expansion} of a collection of sets $\{X_{\gamma}: \gamma \in \Gamma\}$ in $X$
%if $X_{\gamma} \subseteq U_{\gamma} \subseteq X$ for every index $\gamma \in \Gamma$.
%\item[\rm (b)] A collection of sets $\{U_{\gamma} \subseteq X: \gamma \in \Gamma\}$ is called {\it point-finite} if for every $x \in X$
%there are at most finitely many indexes $\gamma \in \Gamma$ such that $x \in U_{\gamma}$. 
%\end{enumerate}

\begin{theorem}{\rm (\cite{KL})}\label{Theor:description}
For a Tychonoff space $X$, the following conditions are equivalent{\rm:}
\begin{enumerate}
\item[{\rm (1)}] $C_p(X)$ is distinguished.
\item[{\rm (2)}] Any countable disjoint collection of subsets of $X$ admits a point-finite open expansion in $X$.
\item[{\rm (3)}] $X$ is a $\Delta$-space.
\end{enumerate}
\end{theorem}

\begin{deff}{\rm (\cite{KL}, \cite{Knight})}\label{def:Delta}
A topological space $X$ is said to be a $\Delta$-space if for every decreasing sequence $\{D_n: n \in \omega\}$
of subsets of $X$ with empty intersection, there is a decreasing sequence $\{V_n: n \in \omega\}$ consisting of open
subsets of $X$, also with empty intersection, and such that $D_n \subseteq V_n$ for every $n \in \omega$. 

The class of all $\Delta$-spaces in the paper is denoted by $\Delta$. Also, any uncountable subset of  the real line $\R$ that is a $\Delta$-space is called a $\Delta$-set. 
\end{deff}

The original definition of a {\em $\Delta$-set} is due to G. M. Reed and E. K. van Douwen (see \cite{Reed}) and the formulation of and 
the interest in $\Delta$-sets followed from the study of the dividing line between the class of normal spaces and the class of countably paracompact spaces.
Research on countably paracompact spaces has a long history in set-theoretic topology related to both the Normal Moore Space Conjecture 
and the study of preservation of normality in products.

C. H.\ Dowker proved that for a normal space $X$ the product $X \times [0, 1]$ is normal if and only if $X$
 is also countably paracompact (equivalently, countably metacompact). 
M. E.\ Rudin  \cite{Rudin} gave the first ZFC example of a Dowker space of size $(\aleph_{\omega})^{\aleph_0}$, but 
Z.\ Balogh \cite{Balogh_Dowker} constructed the first ZFC Dowker space with cardinality of the continuum $\cont$.
Using PCF theory, M. Kojman and S. Shelah \cite{KS} constructed another ZFC Dowker space of cardinality $\aleph_{\omega +1}$.

So, when a normal space is countably paracompact is an important question still attracting attention to this day, and M. E. Rudin's problem whether 
there is a normal not countably paracompact space of size $\omega_1$ in ZFC is still open (see \cite{HM}).
 Conversely, countably paracompactness also does not imply normality but distinguishing between these properties is often delicate.

An uncountable set $X \subset \R$ is called a {\em $Q$-set} if every subset of $X$ is relative $G_\delta$ (see e.g. \cite{FM}).
It is straightforward to show that every $Q$-set is a $\Delta$-set.  
The question whether there can be a $Q$-set goes back to F. Hausdorff and W. Sierpi\'nski in the 1930's (\cite{Hausdorff}, \cite{Sierpinski}) 
and renewed interest in $Q$-sets arose in connection to constructions of interesting counterexamples in the theory of normal spaces.
 F. B. Jones proved that well-known construction of the tangent disc topology over a set of reals $X$ is normal and non-metrizable if and only if $X$ is a $Q$-set.
 As well, the same construction is countably paracompact if and only if $X$ is a $\Delta$-set. 
Furthermore, T. C. Przymusi\'{n}ski
showed that the existence of a $\Delta$-set is actually equivalent to the existence of a separable countably paracompact non-metrizable Moore space \cite[Theorem 7]{P1}.

It is not difficult to show that the existence of a $Q$-set implies $2^{\aleph_0}=2^{\aleph_1}$ \cite{Rothberger},
while it is consistent that there are no $\Delta$-sets (e.g. assuming the Continuum Hypothesis (CH)).
On the other hand, under Martin's Axiom (MA) with the negation of $CH$,
 every uncountable subset of $\R$ of cardinality less than $\cont$ is a $Q$-set (see \cite{FM}).  

Whether it is consistent that there is a $\Delta$-set which is not a $Q$-set seems to be still not well understood. 
The proof appearing in \cite{Knight} is not clear and we are unaware of any expert in the field who understands it.
Let us state as an open and challenging problem to find an alternative proof.
The following question also seems to be open until now.

\begin{problem} {\rm (\cite{Reed2})} \label{problem_card}
Does the existence of a $\Delta$-set imply $2^{\aleph_0}=2^{\aleph_1}$?
 \end{problem}

 Most of the above facts about $Q$-sets and $\Delta$-sets can be found in \cite{Miller}, \cite{Reed}. 

%We should mention that topological spaces $X$ satisfying the above Definition \ref{def:Delta} have been referred to in the literature as $\Delta$-sets, 
%but in this paper, for general topological spaces satisfying Definition \ref{def:Delta} we reserve the term {\em $\Delta$-space}. 
%So, a $\Delta$-set is a $\Delta$-space that is a subset of the real line $\R$.
To distinguish between sets of reals and general topological spaces, we define a topological space $X$ to be a {\em $Q$-space} if every subset of $X$ is a $G_{\delta}$-set
 (equivalently, every subset of $X$ is a $F_{\sigma}$-set). $X$ is called a {\em $\sigma$-closed discrete} space if $X$ is a countable union of closed discrete subspaces.
 We mention the following simple facts about $Q$-spaces: 

\begin{itemize}
\item $X$ is a $Q$-space implies that $X$ is a $\Delta$-space (the proof is the same as for sets of reals, see e.g. \cite[Proposition 4.1]{KL}).
\item If $X$ is a $\sigma$-closed discrete then $X$ is a $Q$-space, therefore also a $\Delta$-space.
\item If $X$ is a $Q$-space then $X$ is a $\sigma$-discrete iff $X$ is $\sigma$-closed discrete.
\end{itemize}

Whether there is a $Q$-space that is not $\sigma$-discrete is quite non-trivial. 
This question was first asked by H. Junnila \cite{Junnila} and led Z. Balogh to define a topological space $X$ to be a {\em $Q$-set space} if $X$ is a $Q$-space and 
 $X$ is not $\sigma$-discrete \cite{Balogh1}. In that paper he gave a beautiful ZFC construction of a $Q$-set space and later improved this result to obtain 
 a hereditarily paracompact, perfectly normal $Q$-set space $X$ such that $|X|= \cont$ \cite{Balogh}. It is worth remarking that the techniques used in constructing
 $Q$-set spaces Z. Balogh later adapted to construct several  Dowker spaces in ZFC including the aforementioned example of size $\cont$.

A systematic study of the class $\Delta$ was originated in the paper \cite{KL} and continued in \cite{KL2}.
Among other results, it was proved that a $\Delta$-space $X$ must be scattered if $X$ is
locally compact \cite{KL} or countably compact \cite{KL2}.
%Recall that a topological space $X$ is {\it scattered} if every closed $A \subseteq X$ has an isolated (in $A$) point.
%In this paper we identify the ordered space of ordinals $[0, \xi)$ with $\xi$. 
It has been shown that the ordered space of ordinals $[0, \omega_1]$ provides
an example of a scattered compact space which is not in $\Delta$ \cite{KL}. 

However, a collection of compact $\Delta$-spaces is quite rich and includes both scattered Eberlein compact spaces and 
one-point compactifications of Isbell-Mr\'owka $\Psi$-spaces (for details see \cite{KL}).

The class $\Delta$ is preserved under closed continuous images and it is invariant under the operation of taking countable unions of closed subspaces \cite{KL2}.
In particular, any topological space which can be represented as a countable union of scattered Eberlein compact spaces is in $\Delta$.

%We say that a topological space $X$ is {\it $\sigma$-closed discrete} if $X = \bigcup_{n\in\omega}X_n$, where each $X_n$ is a closed and discrete subset of $X$.

%It is easy to see that every $\sigma$-closed discrete space is in $\Delta$ and every $\sigma$-discrete $Q$-space is $\sigma$-closed discrete.

It is worth mentioning also that the class $\Delta$ is invariant under taking arbitrary subspaces, and adding a finite set 
to a $\Delta$-space does not destroy the property of being a $\Delta$-space \cite{KL}.  

%Recall that topological space $X$ is {\it countably metacompact} if every countable open cover of $X$ has a point-finite open refinement, or, equivalently,
%if for every decreasing sequence $\{D_n: n \in \omega\}$
%of closed subsets of $X$ with empty intersection, there is a decreasing sequence $\{V_n: n \in \omega\}$ consisting of open
%subsets of $X$, also with empty intersection, and such that $D_n \subset V_n$ for every $n \in \omega$.  
%Immediately, every $\Delta$-space is hereditarily countably metacompact.

The aim of this article is to better understand the boundaries of the class $\Delta$,
by presenting new examples and counter-examples.

We focus on the three very well known set-theoretical constructions producing locally compact first-countable topological spaces:
 $\omega_1$-trees equipped with the interval topology (Section \ref{Section2});
$\Psi(D,\acal)$ spaces, where $\acal$ is an almost disjoint family of countable subsets of an infinite set $D$ (Section \ref{Section3});
and the ladder system spaces $X_L$ (Section \ref{Section4}).  

Summing up several known facts, we note that every almost Souslin tree is countably metacompact (Fact \ref{fact5}), and we observe that 
consistently there is an almost Souslin Aronszajn tree that is a $Q$-space (Fact \ref{fact2}).
The following statement proved in Section \ref{Section2} is one of the main results of our paper: 
 no Souslin tree is a $\Delta$-space (Theorem \ref{Th:Souslin}).

In Section \ref{Section3} we give the negative answer to \cite[Problem 5.11]{KL}:
there is a locally compact $\Psi$-space $X$ such that its one-point compactification is a scattered compact space with a finite Cantor-Bendixson rank,
 but $X \notin \Delta$ (Corollary \ref{cor5.9}).
Also, we demonstrate that there exists an Isbell-Mr\'owka $\Psi$-space $X$ such that one-point extension $X_p = X \cup \{p\}$ of $X$ has uncountable 
 tightness at the point $p$, for some $p \in \beta(X) \setminus X$ (Theorem \ref{example:Dow}). 
This result resolves \cite[Problem 2.15]{Lei_Tkachenko} and \cite[Question 4.1]{Lei_Tkachuk}.

It is consistent with CH that all ladder system spaces on $\omega_1$ are in $\Delta$.
The main result presented in Section \ref{Section4} says that
in forcing extension obtained by adding one Cohen real, there is a ladder system $L$ on $\omega_1$ whose corresponding space $X_L$ 
is not countably metacompact and hence $X_L \notin \Delta$ (Theorem \ref{Theor:Paul_new}).
Also, assuming $MA(\omega_1)$, we prove that for any pair of ladder systems, $L_1$ and $L_2$, the product $X_{L_1}\times X_{L_2}$ is hereditarily normal
(Theorem \ref{Th:prod}).

In the last Section \ref{Section1} we notably generalize several results from \cite{FKLS} and also we 
provide an example of a Lindel\"of $P$-space $X \notin \Delta$ answering negatively \cite[Problem~24]{FS}.
%Our notation and topological terminology are  standard  and  follow  the  book \cite{Engelking}. 
%The most of notions in use are defined explicitly in the text.
%%%%%%%%%%%%%%%%%%%%%%%%%%%%%%%%%%%
\section{Auxiliary notions}\label{Notions}
%%%%%%%%%%%%%%%%%%%%%%%%%%%%%%%%%%
Our set-theoretic notation is standard and follows the text \cite{Kunen}.
Topological terminology follows the book \cite{Engelking}.
For the reader's convenience we recall some relevant notions.
\begin{itemize}
    \item[\ding{192}] A collection of sets $\{U_{\gamma}: \gamma \in \Gamma\}$ is called an {\it expansion} of a collection of sets $\{X_{\gamma}: \gamma \in \Gamma\}$ in $X$ if $X_{\gamma} \subseteq U_{\gamma} \subseteq X$ for every index $\gamma \in \Gamma$.
		A collection of sets $\{U_{\gamma} \subseteq X: \gamma \in \Gamma\}$ is called {\it point-finite} if for every $x \in X$
there are at most finitely many indexes $\gamma \in \Gamma$ such that $x \in U_{\gamma}$.
    \item[\ding{193}] A topological space $X$ is called {\it countably metacompact} if every countable open cover of $X$ has a point-finite open refinement, or, equivalently, if for every decreasing sequence $\{D_n: n \in \omega\}$ of closed subsets of $X$ with empty intersection, there is a decreasing expansion 
		$\{V_n: n \in \omega\}$ consisting of open subsets of $X$, also with empty intersection. 
			
Immediately, every $\Delta$-space is hereditarily countably metacompact.
		\item[\ding{194}] A topological space $X$ is called {\it scattered} if every closed $A \subseteq X$ has an isolated (in $A$) point.
    \item[\ding{195}] Let $p$ be a point of a topological space $X$. We say that $X$ has countable {\it tightness} at $p$, and we write $t(p,X)=\aleph_0$ if 
		whenever $p$ is in the closure $\cl{Y}$ of a set $Y \subseteq X$, then there is a countable $B \subseteq Y$
		such that $p$ is in the closure of $B$.
		Tightness of $X$, $t(X) = \aleph_0$ if $t(p,X) = \aleph_0$ for each point $p\in X$. 
		\item[\ding{196}] A topological space $X$ is called {\it Fr\'echet-Urysohn} if for every subset $Y\subseteq X$ and each $x\in\cl{Y}$ 
		there exists a sequence $\{x_{n}: n\in \omega\}$ in $Y$ which converges to $x$.
    \item[\ding{197}] A Tychonoff space $X$ is called {\it pseudocompact} if every continuous function $f: X \to \R$ is bounded.
		\item[\ding{198}] A topological space $X$ is called a {\it $P$-space} if countable intersections of open sets in $X$ are open.
		\item[\ding{199}] A compact space which can be embedded into a Banach space equipped with the weak topology is called an {\it Eberlein} compact space.
    \end{itemize}
		
In the paper we identify the ordered space of ordinals $[0, \xi)$ with $\xi$.		

%%%%%%%%%%%%%%%%%%%%%%%%%%%%%%%%
\section{Trees}\label{Section2}
%%%%%%%%%%%%%%%%%%%%%%%%%%%%%%%%%
A {\em tree} is a well-founded partially ordered set such the set of predecessors of any element is well-ordered. 
In the article, $\Lev_\alpha(T)$ denotes the $\alpha^{\text{th}}$ level of a tree $T$ and $T_\alpha=\bigcup_{\beta<\alpha} \Lev_\beta(T)$.
An {\em Aronszajn tree} is an $\omega_1$-tree with no uncountable branches. A {\em Souslin tree} is an $\omega_1$-tree such that
 every branch and every antichain is at most countable. 
An $\omega_1$-tree $T$ is said to be {\em almost Souslin} if whenever $A$ is an antichain of $T$ then 
$\{\alpha \in \omega_1: A \cap \Lev_\alpha(T) \neq \emptyset \}$ is not stationary.
A tree is called {\em special} if it can be represented as a countable union of antichains.

 While there are a number of interesting topologies that can be considered on a tree, we consider trees as topological spaces equipped with the interval topology.
The {\em interval topology} on a tree $T$ is generated by the open base consisting of all intervals of the form
$(s,t]=\{x\in T: s <x \leq t\}$, together with all singletons $\{t\}$, where $t$ is a minimal element of $T$. 

Since $\omega_1 \notin \Delta$, if a tree is a $\Delta$-space then it does not have an uncountable branch.
 The following observations are based on the known results from the literature.

\begin{fact}\label{fact1}
Every special tree is $\sigma$-closed discrete \cite{Nyikos}. It follows that every special tree is a $Q$-space, and so every special tree is a $\Delta$-space.  
\end{fact}

\begin{fact}\label{fact2} 
It is consistent with ZFC that there is a non-special Aronszajn tree $T$ that is a $Q$-space, hence a $\Delta$-space. 
\end{fact}

Almost Souslin tree clearly cannot be special.
We claim that it is consistent to have an almost Souslin Aronszajn tree $T$ such that every subset of $T$ is $F_{\sigma}$ (this claim appears in \cite[page 273]{Stevo}).

Let us give a bit more details about this Aronszajn tree. We thank S. Todor\v cevi\'c for pointing out that such an Aronszajn tree can be obtained as follows. 
$\mathbb Q$ denotes the set of rationals endowed with the usual linear order.
Let $w{\mathbb Q}$ be the tree of well-ordered subsets of ${\mathbb Q}$ ordered by end-extension. Clearly, $w{\mathbb Q}$ is a tree of height $\omega_1$.
 It is straightforward to check that the interval topology on $w{\mathbb Q}$ is finer than the natural separable metric topology which 
$w{\mathbb Q}$ inherits from the product topology on $2^{\mathbb Q}$.
 Indeed, let $U$ be a basic open subset of $2^{\mathbb Q}$ determined by the finite partial function $r$, and let $t\in w{\mathbb Q}$ be an element of $U$.
Denote by $q$ the smallest element of $t$ such that $(-\infty,q)\cap \dom(r)=(-\infty, \sup(t))\cap \dom(r)$.
 Then for $s=t\cap (-\infty,q)$, we have $(s,t]\subseteq U$. This means that the interval topology on $w{\mathbb Q}$ has a weaker separable metric topology. 

Recall that $\p$ is the minimal cardinality of a centered family $\fcal$ of infinite subsets of $\omega$ for which one cannot find
an infinite set $A\subseteq \omega$ such that $A\setminus B$ is finite for all $B\in \fcal$ (see \cite{vanDouwen}).

By the Rothberger-Silver theorem, any separable metric space of size $\omega_1$ is a $Q$-space under the assumption $\p > \omega_1$ (see \cite[Corollary 23B]{Fremlin}).
Hence, assuming ${\mathfrak p} > \omega_1$, any subspace of $w{\mathbb Q}$ of size $\omega_1$ is a $Q$-space because it admits a weaker separable metric topology.
 Theorems 15 and 16 of \cite{Stevo2} describe the required almost Souslin Aronszajn subtree $T$ of $w{\mathbb Q}$ which is consistent with $\p > \omega_1$. 

\begin{fact}\label{fact3} 
W. Fleissner, assuming $\bDiamond^{+}$, proved that there exists an Aronszajn tree which is not countably metacompact \cite{Fleissner}.
In \cite{Hanazawa} it was shown that such an example can be constructed from a weaker assumption $\bDiamond$. 
Therefore, it is consistent with ZFC that there exists an Aronszajn tree which is not a $\Delta$-space.
\end{fact}

\begin{fact}\label{fact4}
If a tree $T$ is a $Q$-space, then $T$ is $\R$-embeddable (follows from \cite[Theorem 2.1]{Hart}),
 and since for any $\R$-embeddable $\omega_1$-tree $T$ every uncountable subset of $T$ must contain an uncountable antichain \cite{Devlin},
 no Souslin tree can be a $Q$-space.
\end{fact}

\begin{fact}\label{fact5}
Every almost Souslin tree is (hereditarily) countably metacompact (this result is due to P. Nyikos, see \cite{Fleissner}).
\end{fact}

Below we present one of the main results of our paper. 

\begin{theorem}\label{Th:Souslin}  No Souslin tree is in the class $\Delta$. 
\end{theorem}

\begin{proof}
Let $T$ be a Souslin tree. For a subset $A\subseteq T$ and $B\subseteq A$ we say that $B$ is {\em predense in $A$} if for every $t\in A$ there is $s\in B$ such that $s$ and $t$ are comparable. 
Note that if $B$ is predense in $A$ and if $C\subseteq B$ is a maximal in $B$ antichain, then $C$ is also maximal in $A$. Note also that in a Souslin tree $T$, for any $A\subseteq T$ there is $\alpha<\omega_1$ such that $A\cap T_\alpha$ is predense in $A$. 

\begin{lemma}\label{lem1} If $S\subseteq\omega_1$ is stationary and $A\subseteq T$ is uncountable, then there is some $\gamma\in S$ such that
 $\cl{A}\cap \Lev_\gamma(T)\neq\emptyset$.
\end{lemma}

\begin{proof} Since the set $T_\beta$ is countable for every $\beta<\omega_1$ and $A$ is uncountable,
  $A\setminus T_\beta$ with the order inherited from $T$ is a Souslin tree. Let 
$$C=\{\alpha\in \omega_1: \forall \beta<\alpha\,\,\, (A\cap T_\alpha)\setminus T_\beta\text{ is predense in }A\setminus T_\beta\}.$$
Observe first that $C\subseteq \omega_1$ is closed. Second, $C$ is unbounded.
To see this, let $f:\omega_1\rightarrow \omega_1$ be defined by $$f(\beta)=\min\{\alpha>\beta: (A\cap T_\alpha)\setminus T_\beta\text{ is predense in }A\setminus T_\beta\}.$$
Then $f(\alpha) \in C$ whenever $\alpha \in C$. Thus, $C$ is a club set. 

Next we show that for every $\alpha\in C$ we have $\cl{A}\cap \Lev_\alpha(T) \not=\emptyset$ which completes the proof of the lemma. 
To see this, let $\alpha_n$ be an increasing sequence cofinal in $\alpha$ and for each $n$ let $A_n$ be a maximal antichain in $A\setminus T_{\alpha_n}$ such that $A_n\subseteq T_\alpha$ (we can do this since $\alpha\in C$). 
Then choosing $t\in A\setminus T_\alpha$, and $s\in \Lev_{\alpha}(T)$ such that $s\leq t$ we get that $\bigcup_n A_n$ is cofinal below $s$ and so $s\in \cl{A}$ as required.
\end{proof}

Now we finish the proof that $T$ cannot be a $\Delta$-space. Let $\{S_n\subset\omega_1: n\in\omega\}$ be any countable family of pairwise disjoint stationary sets and let
 $T_n=\bigcup_{\alpha\in S_n} \Lev_\alpha(T)$. By way of contradiction, suppose that there are open sets $U_n\supseteq T_n$ such that the family $\{U_n:n\in \omega\}$ is point-finite. 

For every $t\in T$ define  $\ord(t)=\{k\in\omega:t\in U_k\}$. Choose $n$ so that $A=\{t:n\not\in \ord(t)\}$ is uncountable. 
Now, using Lemma \ref{lem1}, we find $\gamma\in S_n$ and $s\in \Lev_\gamma(T)$ such that $s\in \cl{A}$. 
Thus, since $s\in T_n$ we have that $U_n\cap A\not=\emptyset$. This is a contradiction with $n\notin ord(t)$ for all $t\in A$. 
\end{proof}

Analyzing this proof we see that we used only the following property of a Souslin tree $T$: If $T$ is written as a union of countably many closed subsets,
 one of those subsets must intersect a club set of levels. 
We conjecture that there should be some clearer characterization of when trees are $\Delta$-spaces: 

\begin{problem} \label{problem4}
Find a characterization of trees which are $Q$-spaces / $\Delta$-spaces.
\end{problem}

For a partially ordered set $E$, $\sigma E$ denotes the set of all bounded well-ordered subsets of $E$.
 The order on $\sigma E$ is defined as usual: $s \leq t$ iff $s$ is an initial segment of $t$ \cite{Stevo1}.
Our aim is to consider the above Problem \ref{problem4} for trees of the form $\sigma E$. 

In \cite{Nyikos}, the following statement is claimed:  
($\ast$) $\sigma\Q$ is a ZFC example of an $\R$-embeddable tree which is not countably metacompact (therefore, $\sigma\Q$ is not a $\Delta$-space).
%This claim appears on pages 24 and  26 (arXiv version of Nyikos' paper, dated by 31 Dec, 2004). 
Unfortunately, the proof of this claim ($\ast$) has still not been published, nor could we independently
verify the validity of the statement.

However, under the conjecture that $\sigma\Q$ is not a $\Delta$-space, we would obtain the following (provisional) result.

\begin{conjecture}\label{Theor:sigma-E} Let $E$ be a linearly ordered set. Then the following conditions are equivalent{\rm:}
\begin{enumerate}
\item[{\rm (1)}] Neither $\omega_1$ nor rationals $\Q$ as linearly ordered sets are contained in $E$.
\item[{\rm (2)}] $\sigma E$ is a special tree.
\item[{\rm (3)}] $\sigma E$ is a $Q$-space.
\item[{\rm (4)}] $\sigma E$ is a $\Delta$-space.
\end{enumerate}
\end{conjecture}
%\begin{proof}
(1) $\Leftrightarrow$ (2) is known \cite{Stevo1}. (2) $\Rightarrow$ (3) $\Rightarrow$ (4) have been mentioned before.
Assume (4). Then $E$ does not contain $\omega_1$ because $\omega_1$ is not a $\Delta$-space.
 And assuming that the claim ($\ast$) is true, we would conclude that $E$ could not contain a copy of $\Q$.
This finishes the proof of (4) $\Rightarrow$ (1).
%\end{proof}
\begin{problem} \label{problem5}
Prove or disprove the claim ($\ast$): $\sigma\Q$ is not countably metacompact.
\end{problem}

%%%%%%%%%%%%%%%%%%%%%%%%%%%%% 
\section{MAD families}\label{Section3}
%%%%%%%%%%%%%%%%%%%%%%%%%%%%%
Throughout this section $D$ is an infinite set and $\acal$ is an almost disjoint (AD) family
of countable subsets of $D$. A topological space $\Psi(D,\acal)$ is defined in a standard way:
the underlying set of $\Psi(D,\acal)$ is $D \bigcup \acal$, the points of $D$ are isolated and 
a base of neighborhoods of $A \in \acal$ is the collection of all sets of the form
$\{A\} \cup B$, where $A \setminus B$ is finite. It is known that every first-countable locally compact space
in which the derived set is discrete  is homeomorphic to some $\Psi(D,\acal)$.
Every $\Psi(D,\acal)$ is a Tychonoff space because it is Hausdorff and zero-dimensional.
Note also that all spaces $\Psi(D,\acal)$ are scattered with the Cantor-Bendixson rank equal to $2$. 

Observe that every subset of $\acal$ is closed in $\Psi(D,\acal)$. Further, since $D$ consists of isolated points of $\Psi(D,\acal)$
we obtain the following easy fact:

\begin{proposition} \label{prop_cmc} Let $Y$ be any space $\Psi(D,\acal)$.
Then $Y\in \Delta$ if and only if $Y$ is a countably metacompact space.
\end{proposition}

If $D$ is the set of natural numbers $\N$ and $\acal$ is a maximal almost disjoint (MAD) family of subsets of $\N$, then
$\Psi(D,\acal)$ is called the Isbell-Mr\'owka space and is denoted by $\Psi(\acal)$. 
%Note that a standard Zorn's Lemma argument implies that any almost disjoint family of countable subsets of an infinite set $D$ can be extended to a maximal almost disjoint family on $D$. The cardinal invariant ${\mathfrak a}$ is defined as the minimal size of an almost disjoint family on $\N$. 

\begin{theorem}{\rm (\cite{KL})} \label{Theor:height_2}
Let $\acal$ be any uncountable almost disjoint family of subsets of $\N$.
Denote by $X$ the one-point compactification of the space $\Psi(\N,\acal)$.
 Then $X \in \Delta$, while $X$ is not an Eberlein compact space.
\end{theorem}

Let $D$ be $\omega_1$ and $\acal$ be a MAD family of countable subsets of $\omega_1$.
It is unknown whether the space $\Psi(D,\acal)$ in this case can be countably metacompact (therefore, a $\Delta$-space)
for some MAD family $\acal$ in some model of ZFC \cite{Burke}. 

Recall that the cardinal invariant $\a$ is defined as the minimal size of a MAD family on $\N$ and 
$\b$ denotes the bounding number (see \cite{vanDouwen}).
The details about Theorems \ref{Theor:Dennis1} and \ref{Theor:Paul1}
below can be found in \cite{Burke}, \cite{Paul_SZ}.

\begin{theorem}{\rm (\cite{Burke})} \label{Theor:Dennis1}
Assume $\a =\cont$. Then for every MAD family of countable subsets of $D=\omega_1$ the space 
$\Psi(D,\acal)$ is not countably metacompact.
Therefore, $\Psi(D,\acal) \notin \Delta$.
\end{theorem}

\begin{theorem}{\rm (\cite{Paul_SZ})} \label{Theor:Paul1}
If $\cont=\aleph_2$ or if $\b^{+}=\cont$, then for every MAD family of countable subsets of $D=\omega_1$ the space 
$\Psi(D,\acal)$ is not countably metacompact.
Therefore, $\Psi(D,\acal) \notin \Delta$.
\end{theorem}

A sketch of the construction of the following example was given in \cite{Burke}.
\begin{example} \label{example1}  (\cite{Burke}) In ZFC there exists a MAD family of countable subsets of $D=\omega_1$
such that the space $\Psi(D,\acal)$ is not countably metacompact.

For completeness sake we describe the construction. Identify $\omega_1$ with a subset of $[0,1]$ via an injection $j:\omega_1\rightarrow [0,1]$. Let ${\mathcal A}$ be a MAD family
 of countable subsets of $\omega_1$ so that for each $a\in {\mathcal A}$, $j(a)$ is a convergent sequence in $[0,1]$.  Note that for any subset $S\subset\omega_1$, if $\cl{j(S)}$ in $[0,1]$
 has size ${\mathfrak c}$, then $\cl{S}$ in $\Psi({\mathcal A})$ also has size ${\mathfrak c}$. Since any uncountable subset of $[0,1]$ has closure of size $\cont$,
 it follows that the MAD family $\acal$ has the following key property: For any uncountable subset $S\subseteq \omega_1$ there is a countable $Y\subseteq S$ such that 
the set $\{a\in \acal: |a\cap Y|=\aleph_0\}$ has size $\cont$.

We use this property to define a partition of $\acal=\bigcup_{n\in\omega}\acal_n$ as follows. Enumerate all countable subsets of $\omega_1$ that have uncountable closure in $\Psi(\acal)$
 as $\{Y_\alpha:\alpha<\cont\}$. Recursively choose distinct elements $\{a_{\alpha,n}:\alpha<\cont, n\in \omega\}$ such that for each $\alpha<\cont$ and $n\in\omega$,
the set  $a_{\alpha, n} \cap Y_\alpha$ is infinite. Then fix the partition $\acal=\bigcup_{n\in\omega}\acal_n$ so that the following inclusion holds:
$\{a_{\alpha,n}:\alpha<\cont\}\subseteq \acal_n$ for all $n\in\omega$. 

\begin{claim} The partition $\{\acal_n: n\in \omega\}$ has no point-finite open expansion. 
\end{claim}
\begin{proof} By way of contradiction, suppose that there is a point-finite open family $\{U_n: n\in \omega\}$ such that $U_n\supseteq \acal_n$.
 For each $\gamma\in \omega_1$ let $k_\gamma$ be the maximal index $k$ with $\gamma\in U_k$. Then there are an uncountable $S\subset\omega_1$ and $k\in\omega$ such that
 $k_\gamma=k$ for all $\gamma\in S$. 
And there is a countable $Y\subset S$ with uncountable closure in $\Psi(\acal)$. This $Y$ was enumerated as $Y_\alpha$. Observe that $Y_\alpha\cap U_n=\emptyset$ for all $n > k$,
by definition of $S$. However, by choice of the partition, for every $n$ there is $a\in \acal_n$ with $a\cap Y_\alpha$ infinite.
 For $n>k$ this contradicts that $\{U_n: n\in \omega\}$ is an open expansion of $\{\acal_n: n\in \omega\}$.
The obtained contradiction means that $\Psi(D,\acal)$ is not countably metacompact. 
\end{proof}
\end{example}

As an immediate consequence  we obtain the negative answer to open Problem 5.11 posed in \cite{KL}.

\begin{corollary}\label{cor5.9}
Denote by $X$ the one-point compactification of the locally compact space $\Psi(D,\acal)$ from the above Example \ref{example1}.
Then $X$ is a scattered compact space with the Cantor-Bendixson rank equal to 3, but $X \notin \Delta$.
\end{corollary}

Below we reiterate D. Burke's question from \cite{Burke}: 

\begin{problem} \label{problem6}
Does there exist in ZFC a MAD family of countable subsets of $D$ with $|D|\geq \aleph_1$ such that
$\Psi(D,\acal)$ is a $\Delta$-space?
\end{problem}

We close this section with a $\Psi$-space example answering questions from \cite{Lei_Tkachenko} and \cite{Lei_Tkachuk}.
%Recall that a topological space $X$ has {\it countable tightness}, whenever the following property holds: If a point $x$ is in the closure $\cl{Y}$ of a set $Y \subset X$,
%then there is a countable $B \subset Y$ such that $x$ is in the closure of $B$.
%A topological space $X$ is called {\it Fr\'echet-Urysohn} if for every subset $Y\subset X$ and each $x\in\cl{Y}$ there exists a sequence $\{x_{n}: n\in \omega\}$ in $Y$ which converges to $x$.

It has been shown in \cite{KL2} that every compact $\Delta$-space has countable tightness, and assuming Proper Forcing Axiom (PFA) every countably compact $\Delta$-space is compact and so has countable tightness as well. Whether any countably compact $\Delta$-space has countable tightness in ZFC is an open problem.
 In addition, it has been shown in \cite{Lei_Tkachuk} that pseudocompact $\Delta$-spaces of countable tightness must be scattered.
%Recall that a Tychonoff space $X$ is {\it pseudocompact} if every continuous function $f: X \to \R$ is bounded.
This all obviously gives rise to the following natural problem posed in \cite{Lei_Tkachuk}: 

\begin{problem}\label{pseudocompact} \cite[Question 4.1]{Lei_Tkachuk}
Suppose that $X$ is a pseudocompact $\Delta$-space. Is it true that the tightness of $X$ is countable?
\end{problem}

It has been noted in \cite{Lei_Tkachenko} that for any MAD family $\mathcal{B}$ on $\omega$, the one-point
compactification of the Isbell-Mr\'owka space $\Psi(\N,\mathcal{B})$ is never Fr\'echet-Urysohn. 
Also, there exists an Isbell-Mr\'owka $\Psi$-space $X$ such that $t(p,X_p) = \aleph_0$  
 for every one-point extension $X_p = X \cup \{p\}$ of $X$, where $p \in \beta(X) \setminus X$ and $\beta(X)$
is the Stone-\v{C}ech compactification of $X$ \cite{Lei_Tkachenko}. 

\begin{problem} \label{IM} \cite[Problem 2.15]{Lei_Tkachenko}
 Does there exist an Isbell-Mr\'owka $\Psi$-space $X$ such that $t(p,X_p) > \aleph_0$, for some point $p \in \beta{X} \setminus X$?
\end{problem}  

Now we completely resolve Problems \ref{pseudocompact} and \ref{IM}.

\begin{theorem}\label{example:Dow} There is an Isbell-Mr\'owka $\Psi$-space $X$ with a point $p\in \beta{X} \setminus X$ such that
 $X_p$ is a pseudocompact $\Delta$-space and $t(p,X_p) > \aleph_0$.
\end{theorem}

\begin{proof} In \cite{Dow} a MAD family ${\mathcal M}\subseteq [\omega]^\omega$ is constructed with the property that the reminder of the 
Stone-\v{C}ech compactification $\beta\Psi({\mathcal M})\setminus \Psi({\mathcal M})$ is homeomorphic to the compact ordinals space $\omega_1+1$.

We declare $X = \Psi({\mathcal M})$ and choose $p\in \beta{X} \setminus X$ the point $\omega_1$. Then
$X_p = X \cup \{p\}$ is the desired space.
 To see this, first note that since ${\mathcal M}$ is maximal, $\Psi({\mathcal M})$ is pseudocompact.
 We conclude that $X_p$ is pseudocompact, because $X=\Psi({\mathcal M})$ is dense in $X_p$.
Moreover, $X_p$ is a $\Delta$-space, because adding a finite set does not destroy the $\Delta$-space property. 
In order to show that $X_p$ has uncountable tightness at the point $p$, we rely on some details of the construction.

For the following definitions we refer to \cite{vanDouwen}.
Define the quasi-order $\subseteq^*$ on $\mathscr{P}(\omega)$ by the rule:
 $A \subseteq^* B$ if $A \setminus B$ is finite.
We say that $A$ is a {\it pseudo-intersection} of a family $\mathscr{F}$ if $A \subseteq^* F$ for each $F \in \mathscr{F}$.
We call $\mathscr{T} \subseteq \left[\omega\right]^{\omega}$ a {\it tower} if $\mathscr{T}$ is well-ordered by $\supseteq^*$ and has no infinite pseudo-intersection.

 A MAD family ${\mathcal M}$ in \cite{Dow} is constructed as a union of families of sets ${\mathcal C}_\alpha$ along with a chain
$$
\{T_\alpha:\alpha<\omega_1\}\subseteq [\omega]^\omega
$$
that is $\subseteq^*$-increasing, and
 so that ${\mathcal M}=\bigcup\{{\mathcal C}_\alpha:\alpha\leq \omega_1\}$ has the properties 
\begin{enumerate}
\item For each $\alpha$, $C_\alpha$ is a non-empty family of infinite subsets of $\omega$ such that  ${\mathcal C}_\alpha\subseteq\{x\subseteq T_\alpha:x\cap T_\beta=^*\emptyset\,\,\text{ for all }\,\,\beta<\alpha\}$, and 
\item If $\{T_\alpha:\alpha<\omega_1\}$ is not a tower, then ${\mathcal C}_{\omega_1}\subseteq\{x:x\cap T_\beta=^*\emptyset\,\,\text{ for all }\,\,\beta<\omega_1\}$ is not empty, and if $\{T_\alpha:\alpha<\omega_1\}$ is a tower then ${\mathcal C}_{\omega_1}$ is empty.
\item The closure of $T_\alpha$ in $\beta\Psi({\mathcal M})$ is clopen and is equal to 
$$
T_\alpha\cup\left(\bigcup\{{\mathcal C}_\alpha: \beta\leq \alpha\}\right)\cup [0,\alpha]
$$
\end{enumerate}
Let us comment the item (3). Indeed, the closure of $T_\alpha$ is the set of $x\in {\mathcal M}$ such that $x\cap T_\alpha$ is infinite. 
And if $\gamma>\alpha$ then each $x\in C_{\gamma}$ has finite intersection with $T_\alpha$. In the opposite case, if $\gamma\leq \alpha$ 
then for each $x\in C_{\gamma}$ we have that $x \subseteq^* T_\alpha$. 

With this description, let us now define $Y=\bigcup_{\alpha<\omega_1} {\mathcal C}_\alpha$
 (by $(2)$, $Y$ may be all of ${\mathcal M}$ if ${\mathcal C}_{\omega_1}=\emptyset$).
 Then $[0,\omega_1)$ is contained in the closure of $Y$ in $\beta\Psi({\mathcal M})$ and so $\omega_1$ is in the closure of $Y$ in $X$.
 Moreover, any countable subset of $Y$ is a subset of $\bigcup_{\beta<\alpha} {\mathcal C}_\beta$ for some $\alpha<\omega_1$ and hence, by $(1)$ it is contained in the closure of $T_\alpha$. By $(3)$ the closure of $T_\alpha$ in $\beta\Psi({\mathcal M})$ does not include the point $\omega_1$. And so, it follows that $\omega_1$ is not in the closure of any countable subset of $Y$.
So, $X_p$ is a pseudocompact $\Delta$-space with uncountable tightness at the point $p=\omega_1$.
\end{proof}

%%%%%%%%%%%%%%%%%%%%%%
\section{Ladder systems}\label{Section4}
%%%%%%%%%%%%%%%%%%%
Let $L$ be a {\it ladder system} over a stationary subset of limit ordinals $D\subset \omega_1$. 
I.e. $L=\{s_\alpha:\alpha\in D\}$, where each $s_\alpha$ is an $\omega$-sequence in $\alpha$ cofinal in $\alpha$.
 Any such family is an almost disjoint family of countable subsets of $\omega_1$ and so we can form  the space $\Psi(D,\acal)$, where $\acal = L$.
Traditionally, in this case the resulting space is denoted in a bit different fashion.
Namely, if $L$ is a ladder system on a stationary set $S\subseteq Lim(\omega_1)$, we denote by $X_L$ the topological space 
$\omega_1\times \{0\} \cup S\times \{1\}$, where every point $(\alpha, 0)$ is isolated and for each $\alpha\in S$,
 a basic neighborhood of $(\alpha, 1)$ consists of $\{(\alpha, 1)\}$ along with a cofinite subset of $s_{\alpha}\times \{0\}$.

We ask the following natural question: under which conditions on the ladder system $L$, $X_L$ is a $\Delta$-space? 
This question, in relation to other covering and separation properties of $X_L$, actually was studied already in \cite{BE+}.
As a matter of fact, $X_L\in \Delta$  has been characterized by a certain type of uniformization property imposed on $L$.

It was S. Shelah who introduced the notion of a ladder system $L$ being uniformizable. 
One says that a ladder system is {\it $2$-uniformizable} if for any sequence of functions $\eta_\alpha:s_\alpha\rightarrow 2$ there is a uniformizing function $f:\omega_1\rightarrow 2$, 
meaning that $f\upharpoonright s_\alpha=^*\eta_\alpha$ (is equal for all but finitely many elements of $s_\alpha$) for all $\alpha\in D$.
While this notion was introduced in Shelah's work on Whitehead groups, it is straightforward to show that $X_L$ is normal if and only if every sequence of constant functions admits 
a uniformizing function. Further, $X_L$ is normal implies that $X_L$ is a countably metacompact
 (which in turn is characterized by an even weaker notion of uniformizability \cite[Section 1]{BE+}). 

$MA(\omega_1)$ implies the stronger property of being $\omega$-uniformizable, meaning that any sequence of functions  $\eta_\alpha:s_\alpha\rightarrow \omega$ can be uniformized.
 It is  also worth remarking that if a ladder system $L$ is $\omega$-uniformizable, then the associated space $X_L$ is $\sigma$-closed discrete, 
which of course implies that $X_L$ is a $\Delta$-space. 
 Indeed, if $L=\{s_\alpha:\alpha\in \lim(\omega_1)\}$ is any ladder system, then we can enumerate each $s_\alpha$ in increasing order as $\{s_\alpha(n):n\in\omega\}$.
 Let $\eta_\alpha:s_\alpha\rightarrow\omega$ be given by $\eta_\alpha(s_\alpha(n))=n$. If $f$ uniformizes the family $\{\eta_\alpha:\alpha\in \lim(\omega_1)\}$, then it follows that
the fibers $f^{-1}(n)$ are closed discrete for all $n\in\omega$ and therefore $X_L$ is a $\sigma$-closed discrete space, so they are even $Q$-spaces. 
Summarizing these facts, we note the following
 
\begin{remark} \label{remark:DS} 
$MA(\omega_1)$ implies that every ladder system space $X_L$ is a normal $\sigma$-closed discrete space (see \cite{BE+}, \cite {DS}),
 hence it is consistent that all $X_L$ are in the class $\Delta$.
\end{remark}

\begin{example} \label{example2}
In ZFC there is a ladder system $L$ on $\omega_1$ such that the corresponding space $X_L$ is countably metacompact \cite[Section 1]{BE+},
hence $X_L \in \Delta$. Indeed, any ladder system $L$ with the property that the
$n$-th element of each ladder $s_{\alpha}$ is of the form $\beta + n$ with $\beta$ a limit ordinal leads to 
a $\sigma$-closed discrete space $X_L$.
Therefore, the one-point compactification of the locally compact space $X_L$ provides a new example of 
a scattered compact $X \in \Delta$, while $X$ is not an Eberlein compact space.
\end{example} 

Now we present one of the main results of our paper. It shows that consistently there are $X_L \notin \Delta$.
 
\begin{theorem} \label{Theor:Paul_new}
In forcing extension of ZFC obtained by adding one Cohen real, there is a ladder system $L$ on $\omega_1$ whose corresponding space $X_L$ 
is not countably metacompact and hence $X_L \notin \Delta$.
\end{theorem}
\begin{proof}
 First we describe how to construct a ladder system coded by a single prescribed function $g:\omega\rightarrow \omega$. 
For each limit countable ordinal $\alpha$ fix a bijection $e_\alpha:\alpha\rightarrow \omega$ 
and choose arbitrarily an increasing sequence $$\alpha(0)<...<\alpha(n)<...<\alpha \text{ cofinal in }\alpha.$$

For each limit countable ordinal $\alpha$ we define $s^g_\alpha=\{\beta^g_n:n\in \omega\} \subset \omega_1$ as follows.
Let $n \in \omega$.  Consider the set of all $\beta\in [\alpha(n),\alpha(n+1))$ such that $g(e_\alpha(\beta)) > n$. 

If this set is not empty, we declare that $\beta^g_n$ is the unique such $\beta$ with $e_\alpha(\beta)$ minimal. Otherwise we decide $\beta^g_n=0$. 
Put $S_g=\{\alpha\in Lim(\omega_1): s^g_\alpha\text{ is cofinal in }\alpha\}$. Then $S_g$ is a stationary set and $L_g=\{s^g_\alpha:\alpha\in S_g\}$ is a ladder system on $S_g$. 

Now we keep $\{e_\alpha:\alpha\in Lim(\omega_1)\}$ and sequences $\alpha(n)$ for each $\alpha\in Lim(\omega_1)$ fixed in the ground model $V$ and force with ${\mathbb P}=\text{Fn}(\omega,\omega)$.
 Let $G$ be ${\mathbb P}$ generic over $V$ and let $c=\bigcup G:\omega\rightarrow \omega$ be the generic function. 
The proof of the following claim is based on a standard density argument and only uses that for any infinite $A\subseteq\omega$ in the ground model, the image $c(A)$ is infinite. 

\begin{lemma} $S_c=\{\alpha\in Lim(\omega_1): s^c_\alpha\text{ is cofinal in }\alpha\}=Lim(\omega_1)$. 
\end{lemma}
%\hfill$\square$
%\vskip 6pt

Our aim is to prove that $\Vdash X_{L_c}$ is not countably metacompact. On the contrary, assume that  $X_{L_c}$ is countably metacompact.
Choose in $V$ any partition $Lim(\omega_1)=\bigcup_{n\in\omega} S_n$ into pairwise disjoint stationary sets.
Form the corresponding disjoint family consisting of closed sets $\{S_n\times\{1\}: n\in\omega\}$.
Fix names $U_n$ for the open sets expanding $S_n\times\{1\}$ in the extension. It suffices to assume that $$\Vdash\{U_n:n\in \omega\}\text{ is point-finite }$$ and obtain a contradiction. 

Assuming that this open expansion is forced to be point-finite, we may fix for each ordinal $\alpha\in \omega_1$, an element $q_\alpha\in {\mathbb P}$ and a finite set $F_\alpha\subseteq \omega$ 
such that  $$q_\alpha\Vdash ``(\alpha,0)\in U_n \text{ iff } n\in F_\alpha"$$

By the pigeon hole principle, there is an uncountable $A\subseteq \omega_1$,  an element $q\in P$ and a finite $F\subset \omega$ such that $F_\alpha=F$ and $q_\alpha=q$ for all $\alpha\in A$. 

Pick $m\not\in F$ and consider the stationary set $S_m$. 
In the extension, we have that for each $\beta\in S_m$ there is a $\gamma<\beta$ such that $(s^c_\beta\setminus \gamma)\times\{0\}\subseteq U_m$. So, for each $\beta\in S_m$ we may extend $q$ to $p_\beta$ and fix $\gamma_\beta$ such that 
$$
p_\beta\Vdash (s^c_\beta\setminus \gamma_\beta)\times\{0\}\subseteq U_m
$$
And again by the pigeon hole principle, and the Fodor's lemma, we may fix a single $\gamma$ and a single $p\leq q$ and a stationary $S'_m\subseteq S_m$ such that for every $\beta\in S'_m$
$$
p\Vdash (s^c_\beta\setminus \gamma)\times\{0\}\subseteq U_m
$$

Choose $\beta\in S'_m$ such that $A\cap \beta$ is unbounded in $\beta$ and
choose $k\in\omega$ large enough so that 
\begin{enumerate}
\item $\dom(p)\subseteq e_\beta((0,\beta(k)])$,
\item $k$ above is the maximum of the range of $p$,
\item $A\cap [\beta(k),\beta(k+1))\not=\emptyset$, and
\item $\beta(k)>\gamma$. 
\end{enumerate}
Fix $\eta\in A\cap [\beta(k),\beta(k+1))$ and denote $E=\{\xi\in [\beta(k),\beta(k+1)):e_\beta(\xi)<e_\beta(\eta)\}$. Then $E$ is finite and $e_\beta(E\cup\{\eta\})$ is disjoint from the domain of $p$ (by (1)). So we may extend $p$ to $p'$ defining $p'(e_\beta(\xi))<k$ for all $\xi\in E$ and defining $p'(e_\beta(\eta))> k$. Then, by definition of $s^c_\beta$ we have that $\eta\in s^c_\beta$. And $\eta>\gamma$ (by (4)). Putting this all together we have that 
$$
p'\Vdash \eta\in U_m
$$
However, by definition of $A$, we also have that $p\Vdash \eta\not\in U_m$ since $m\not\in F$.
Finally, in view of $p'\leq p$ this is a contradiction completing the proof. 
%\hfill$\square$
\end{proof}

It is natural to investigate whether the class $\Delta$ is invariant under the basic topological operations.
An account about the progress in this direction can be found in \cite{KL2}, \cite{Lei_Tkachuk}.
In particular, it has been shown that the product of a $\Delta$-space with a $\sigma$-closed discrete space
is a $\Delta$-space \cite[Corollary 2.9]{KL2}. But the general question whether the class $\Delta$ is invariant under finite products remains open even for the compact factors \cite[Problem 5.10]{KL}.

One may ask whether ladder system spaces, or more generally, the spaces $\Psi(D,\acal)$ provide counterexamples to this question.
\begin{proposition}\label{prop:prod} Let $Y$ be any $\Delta$-space. Then
\begin{enumerate}
\item[{\rm (i)}] The product $\Psi(\N,\acal) \times Y$ is a $\Delta$-space.
\item[{\rm (ii)}] Assume $MA(\omega_1)$. Then the product $X_L \times Y$ is a $\Delta$-space for every ladder system $L$ on $\omega_1$.
\end{enumerate}
\end{proposition}
\begin{proof} Both $\Psi(\N,\acal)$ and $X_L$ are $\sigma$-closed discrete spaces.
\end{proof}

Similarly, we observe that all known examples of compact $\Delta$-spaces can be seen to be countable unions of scattered Eberlein compact spaces. 
Since the product of two scattered Eberlein compact spaces is a scattered Eberlein compact space,
the finite products for all compact $\Delta$-spaces we are aware of remain to be $\Delta$-spaces.
This motivates us to reiterate the following open problem.

\begin{problem}\label{prob_countun} \cite[Question 4.6]{Lei_Tkachuk}
Is it true  that  every  compact  $\Delta$-space is the countable union of Eberlein compact spaces? 
\end{problem}

Recall that for the ladder system space $X_L$, normality of $X_L$ implies that $X_L\in\Delta$.
What can be said regarding the normality property of the product $X_{L_1}\times X_{L_2}$ under $MA(\omega_1)$?

\begin{theorem} \label{Th:prod} Assume $MA(\omega_1)$. Then, for any pair of ladder systems, $L_1$ and $L_2$, the product $X_{L_1}\times X_{L_2}$ is hereditarily normal. 
\end{theorem}

\begin{proof} To fix our notation, let $L_1=\{x_\alpha:\alpha\in S\}$  and
$L_2=\{y_\alpha:\alpha\in T\}$ be two ladder systems on stationary sets $S, T\subseteq \text{Lim}(\omega_1)$, respectively.
Note that $X_{L_1}\times X_{L_2}$ is scattered of the height 3 with the closed discrete set $L_1\times L_2$ on the top level in the Cantor-Bendixson decomposition. 

We first prove 
\begin{lemma}\label{L1} Assume $MA(\omega_1)$. Then for any partition $H: L_1\times L_2\rightarrow 2$, the sets $H^{-1}(0)$ and $H^{-1}(1)$ can be separated by disjoint open sets. 
\end{lemma}

\begin{proof} To this end, we take the natural partially ordered set ${\mathbb P}$ of finite partial neighborhood assignments to the points of $L_1\times L_2$ respecting the partition. Namely,

${\mathbb P}$ is the set of $p\in Fn(S\times T,\omega)$ such that $$U_{p(\alpha,\beta)}(x_\alpha)\times U_{p(\alpha,\beta)}(y_\beta)\cap U_{p(\delta,\gamma)}(x_\delta)\times U_{p(\delta,\gamma)}(y_\gamma)=\emptyset,$$ whenever $(\alpha,\beta),(\delta,\gamma)\in \dom(p)\text{ and }H(\alpha,\beta)\not=H(\delta,\gamma)$.

It suffices to show that ${\mathbb P}$ has the ccc. We will show, in fact, that ${\mathbb P}$ has Property K. 
To this end, fix $\{p_\alpha:\alpha\in \omega_1\}\subseteq {\mathbb P}$ and for each $\alpha$, let 
$$
h(\alpha)=\text{max}\{\eta<\alpha: \eta \text{ appears in the domain of }p_\alpha\}
$$
By the Fodor's lemma, there is an $\xi^\prime$ and a stationary $A$ such that $h(\alpha)<\xi$ for all $\alpha\in A$.

Next define $g:A\rightarrow \omega_1$ to be the maximum of the following two maxima 
$$\text{max}\left[\alpha\cap \bigcup \{U_{p_\alpha(\delta,\gamma)}(x_\delta) : (\delta,\gamma)\in \dom(p_\alpha) \delta>\alpha\}\right] $$ and 
$$\text{max}\left[\alpha\cap \bigcup \{U_{p_\alpha(\delta,\gamma)}(y_\gamma) : (\delta,\gamma)\in \dom(p_\alpha) \gamma>\alpha\}\right]
$$
Again by the Fodor's lemma, there is a $\xi>\xi^\prime$ and a stationary $B$ such that $g(\alpha)<\xi$ for all $\alpha\in B$. 

By the pigeon hole principle, we may assume that for all $\alpha$ and $\beta$ in $B$, $p_\alpha$ and $p_\beta$ are isomorphic conditions in the following sense: For each $\alpha$ and $\beta$ in $B$, there is a bijection $j:\dom(p_\alpha)\rightarrow \dom(p_\beta)$ such that 
\begin{enumerate}
\item $j$ is an order isomorphism with respect to the lexicographic ordering on $\omega_1\times\omega_1$, 
\item For all $(\gamma,\delta)\in \dom(p_\alpha)$, $H((\gamma,\delta))=H(j(\gamma,\delta))$ and $p_\alpha(\gamma,\delta))=p_\alpha(j(\gamma,\delta))$
\item Moreover, if $j(\gamma, \delta)=(\gamma',\delta')$ then 
\begin{enumerate}
\item $\gamma<\alpha$ iff $\gamma'<\beta$ in which case both are less than $\xi$ and we require $\gamma=\gamma'$, and $\delta<\alpha$ iff $\delta'<\beta$ in which case both are also less than $\xi$ and we also require $\delta=\delta'$
\item $\gamma=\alpha$ iff $\gamma'=\beta$ and $\delta=\alpha$ iff $\delta'=\beta$ 
\item For $\gamma>\alpha$ then $x_\gamma \cap \alpha = x_{\gamma'}\cap \beta$ and for $\delta>\alpha$ then $y_\delta \cap \alpha = y_{\delta'}\cap \beta$ (note: in this clause if $\gamma>\alpha$ then $x_\gamma\cap \alpha=x_\gamma\cap \xi$ and similarly for $\gamma', \delta, \delta'$).
\end{enumerate}
\end{enumerate}

 If $\beta\in B$, let $\beta^+$ denote the minimum of $B$ above $\beta$. Finally, we apply the Fodor's lemma and pigeon hole principle once again to $B$. Take $C\subset B$ stationary and $F, G$ finite so that for all $\beta\in C$, if $\beta^+\in S$ then $x_{\beta^+} \cap \beta=F$ and if $\beta^+\in T$ then $y_{\beta^+}\cap \beta = G$ (for $\beta\in C$, $\beta^+$ still denotes the minimum of $B$ above $\beta$). 
 
Now, for any $\alpha,\beta\in C$ we have that $p_{\alpha^+}$ and $p_{\beta^+}$ are compatible. 
\end{proof}

%Note that at this stage, Lemma \ref{L1} is already enough to prove that $X_{L_1}\times X_{L_2}$ is a $\Delta$-space (see the corollary and its proof below). 

Next, we show 
\begin{lemma}\label{L2} $MA(\omega_1)$ implies that $L_1\times \omega_1$ and $\omega_1\times L_2$ can be separated by disjoint open sets. 
\end{lemma} 

\begin{proof} To this end, we define another partial ordered set ${\mathbb Q}$ consisting of pairs $(p,q)$ such that $p\in Fn(S\times\omega_1,\omega)$ and $q\in Fn(\omega_1\times T,\omega)$ such that for all $(\alpha,\gamma)\in \dom(p)$ and all $(\delta,\beta)\in \dom(q)$, the open sets $U_{p(\alpha,\gamma)}(x_\alpha)\times \{\gamma\}$ and $\{\delta\}\times U_{q(\delta,\beta)}(y_\beta)$ are disjoint.

Note that since any finite union of basic open sets is clopen, for any $(\alpha,\gamma)\in S\times \omega_1$, the set of $(p,q)\in {\mathbb Q}$ such that $(\alpha,\gamma)\in \dom(p)$ is dense. Similarly, for any $(\delta,\beta)\in \omega_1\times T$, the set of $(p,q)\in {\mathbb Q}$ such that $(\delta,\beta)\in \dom(q)$ is also dense. And any filter generic with respect to 
these dense sets defines a pair of open sets separating as required. So, it suffices to show that ${\mathbb Q}$ is CCC. 

To this end, fix $\{(p_\xi,q_\xi):\xi\in \omega_1\}\subseteq{\mathbb Q}$ and we will find an uncountable subset consisting of pairwise compatible elements. 

First, as we did before, we may thin out our sequence using the Fodor's lemma and the pigeon hole principle to obtain a stationary set $A$ and an $\eta\in \omega_1$ so that the subset $\{(p_\xi,q_\xi):\xi\in A\}$ consists of pairwise isomorphic conditions in the following sense: 

For all $\xi,\chi\in A$ there are bijections $h:\dom(p_\xi)\rightarrow \dom(p_\chi)$ and $g:\dom(q_\xi)\rightarrow \dom(q_\chi)$. Moreover, for every
 $(\alpha,\delta)\in \dom(p_\xi)$, if we denote $h(\alpha,\delta)=(\alpha', \delta')$ then 
\begin{enumerate}
\item $p_\xi((\alpha,\delta)=p_\chi((\alpha',\delta'))$
\item $\alpha<\xi$ iff $\alpha'<\chi$ and in this case $\alpha=\alpha'<\eta$
\item $\delta<\xi$ iff $\delta'<\chi$ and in this case $\delta=\delta'<\eta$
\item $\alpha=\xi$ iff $\alpha'=\chi$ and $\delta=\xi$ iff $\delta'=\chi$. 
\item In the case that $\alpha>\xi $ (and so also $\alpha'>\chi$) we then have $x_\alpha\cap \xi=x_{\alpha'}\cap \chi \subseteq \eta$. 
\end{enumerate}
And, for every $(\gamma,\beta)\in \dom(q_\xi)$, if we denote $g(\gamma,\beta)=(\gamma', \beta')$ then 
\begin{enumerate}
\item $q_\xi((\gamma,\beta)=q_\chi((\gamma',\beta'))$
\item $\beta<\xi$ iff $\beta'<\chi$ and in this case $\beta=\beta'<\eta$
\item $\gamma<\xi$ iff $\gamma'<\chi$ and in this case $\gamma=\gamma'<\eta$
\item $\beta=\xi$ iff $\beta'=\chi$ and $\gamma=\xi$ iff $\gamma'=\chi$. 
\item In the case that $\beta>\xi $ (and so also $\beta'>\chi$) we then have $y_\beta\cap \xi=y_{\beta'}\cap \chi \subseteq \eta$. 
\end{enumerate}
Now we do one final thinning out. For each $\nu\in A$, let $\nu^+$ denote the minimum of $A\setminus (\nu+1)$. Applying the Fodor's lemma one last time, we may fix $B\subseteq A$ stationary, $\eta'\geq \eta$ and $F,G\subseteq \eta'$ finite such that for any $\nu\in B$, if $\nu^+\in S$ (where $\nu^+$ still denotes the minimum of $A$ above $\nu$), then $x_{\nu^+}\cap \nu =F$ (and so $x_{\nu^+}\cap [\eta',\nu)=\emptyset$), and if $\nu^+\in T$ then $y_{\nu^+}\cap \nu =G$ (and so $y_{\nu^+}\cap [\eta',\nu)=\emptyset$). Now by going to a subsequence we may assume that  

\begin{enumerate}
\item[$(*)$] for $\nu<\mu$ in $B$, we have that $\dom(p_{\nu^+})\cup \dom(q_{\nu^+}) \subseteq \mu\times\mu$. 
\end{enumerate}

We now have the following claim, which completes the proof that ${\mathbb Q}$ has the CCC:

%\vskip 6pt

\begin{claim} For all $\nu, \mu\in B$, $(p_{\nu^+},q_{\nu^+})$ is compatible with $(p_{\mu^+},q_{\mu^+})$
\end{claim}
\begin{proof} Suppose not and that $(p_{\nu^+},q_{\nu^+})$ is incompatible with $(p_{\mu^+},q_{\mu^+})$ where $\nu<\mu$.
 Then by symmetry, we may assume that there is a $(\alpha,\delta)\in \dom(p_{\nu^+})\setminus \dom(p_{\mu^+})$ and 
a $(\gamma,\beta)\in \dom(q_{\mu^+})\setminus \dom(q_{\nu^+})$ witnessing incompatibility. This means that
$$
\gamma\in U_{p_{\nu^+}(\alpha,\delta)}(x_\alpha)\text{ and }\delta\in  U_{q_{\mu^+}(\gamma,\beta)}(y_\beta)
$$
Since $(\alpha,\delta)\in \dom(p_{\nu^+})$ we have by $(*)$ that $\alpha<\mu$ and since $\gamma\in x_\alpha$ we then have $\gamma<\mu$. Therefore, by our thinning out we have that $\gamma<\eta$. Now, by the isomorphism between $(p_{\nu^+},q_{\nu^+})$ and $(p_{\mu^+},q_{\mu^+})$, we have that there is a $\beta'$ such that $(\gamma, \beta')\in \dom(q_{\nu^+})$ and $y_\beta\cap \eta=y_{\beta'}\cap \eta$ and $q_{\nu^+}(\gamma,\beta')= q_{\mu^+}(\gamma,\beta)$. 
And now consider that $\delta\in y_\beta$ and $\delta<\mu$ (since it appears in the domain of $p_{\nu^+}$). Thus $\delta<\eta$ and so also $\delta\in y_{\beta'}$. By the isomorphism of conditions, since $\delta\in  U_{q_{\mu^+}(\gamma,\beta)}(y_\beta)$ we also have $\delta\in  U_{q_{\nu^+}(\gamma,\beta')}(y_\beta')$ (this is because $q_{\mu^+}(\gamma,\beta)=q_{\nu^+}(\gamma,\beta')$ and the initial segment of $y_\beta$ up to $\delta$ is the same as the initial segment of $y_{\beta'}$ up to $\delta$).

So, we have that $(\gamma, \beta')\in \dom(q_{\nu^+})$, $(\alpha,\delta)\in \dom(p_{\nu^+})$ 
and $\delta\in  U_{q_{\nu^+}(\gamma,\beta')}(y_\beta')$ and $\gamma\in U_{p_{\nu^+}(\alpha,\delta)}(x_\alpha)$ contradicting that $(p_{\nu^+},q_{\nu^+})$ is an element of ${\mathbb Q}$. 
\end{proof}
\end{proof}

We need one more lemma.
\begin{lemma}\label{L3} Assume $MA(\omega_1)$. If $H\subseteq L_1\times \omega_1$, $K\subseteq L_1\times L_2$ and $\cl{H}\cap K=\emptyset$,
 then $H$ and $K$ can be separated by disjoint open sets. 
\end{lemma}
Remark that by symmetry, the version of  Lemma \ref{L3}, where $H$ is taken as a subset of $\omega_1\times L_2$ also holds and has the same proof. 

\begin{proof}  We take the natural partially ordered set of finite partial neighborhood assignments to the points of $H\cup K$ that approximate a disjoint open assignment. Namely,

${\mathbb P}$ is the set of pairs $(p,q)$ such that $p\in Fn(H,\omega)$ and $q\in Fn(K,\omega)$ such that $$U_{p(\alpha,\beta)}(x_\alpha)\times \{\beta\}\cap U_{q(\delta,\gamma)}(x_\delta)\times U_{q(\delta,\gamma)}(y_\gamma)=\emptyset,$$ 
whenever $(x_\alpha,\beta)\in H\cap \dom(p)$ and $(x_\delta,y_\gamma)\in K\cap \dom(q)$.

As above, it suffices to prove that ${\mathbb P}$ has the CCC and to this end we index an arbitrary uncountable subset of ${\mathbb P}$ as $\{(p_\alpha, q_\alpha):\alpha\in \omega_1\}$. 

By the same thinning out we did in the previous lemmas, we obtain stationary sets $C\subseteq B$, a $\xi$, finite sets $G,J\subseteq \xi$ such that for each $\alpha,\beta\in B$ the conditions $(p_\alpha,q_\alpha)$ and $(p_\beta,q_\beta)$ are isomorphic via a bijection $j$, meaning
\begin{enumerate}
\item For each $(x_\gamma, \eta)\in \dom p_\alpha$, denoting $j((x_\gamma,\eta))\in \dom(p_\beta)$ by $(x_{\gamma'},\eta')$, we have
\begin{enumerate}
\item $\gamma<\alpha$ iff $\gamma'<\beta$ in which case $\gamma=\gamma' <\xi$
\item $\eta<\alpha$ iff $\eta'<\beta$ in which case $\eta=\eta'<\xi$
\item $\gamma=\alpha$ iff $\gamma'=\beta$ and $\eta=\alpha$ iff $\eta'=\beta$
\item $\gamma>\alpha$ iff $\gamma'>\beta$ in which case $x_\gamma\cap \alpha=x_{\gamma'}\cap \beta \subseteq \xi$. 
\end{enumerate}
\item For each $(x_\gamma, y_\eta)\in \dom q_\alpha$, denoting $j((x_\gamma,y_\eta))\in \dom(q_\beta)$ by $(x_{\gamma'},y_{\eta'})$, we have the same type of isomorphism properties:
\begin{enumerate}
\item $\gamma<\alpha$ iff $\gamma'<\beta$ in which case $\gamma=\gamma' <\xi$
\item $\eta<\alpha$ iff $\eta'<\beta$ in which case $\eta=\eta'<\xi$
\item $\gamma=\alpha$ iff $\gamma'=\beta$ and $\eta=\alpha$ iff $\eta'=\beta$ 
\item $\gamma>\alpha$ iff $\gamma'>\beta$ in which case $x_\gamma\cap \alpha=x_{\gamma'}\cap \beta \subseteq \xi$. 
\item $\eta>\alpha$ iff $\eta'>\beta$ in which case $y_\eta\cap \alpha=y_{\eta'}\cap \beta \subseteq \xi$. 
\end{enumerate}
\end{enumerate}
And again, as before, for each $\alpha\in C$, we denote $\alpha^+$ the minimum of $B$ above $\alpha$, and so that 
\begin{enumerate}
\item if $\alpha^+\in S$ then $x_{\alpha^+}\cap \alpha = G$
\item if $\alpha^+\in T$ then $y_{\alpha^+}\cap \alpha = J$. 
\end{enumerate}
And now we can conclude that the family $\{(p_{\alpha^+},q_{\alpha^+}):\alpha\in C\}$ is centered. 
\end{proof}

Finally, to prove that $X_{L_1}\times X_{L_2}$ is hereditarily normal, fix $A$ and $B$ subsets of $X_{L_1}\times X_{L_2}$ such that
 $\cl{A}\cap B=A\cap\cl{B}=\emptyset$. By Lemma \ref{L1} we may fix $U_2$ and $V_2$ disjoint with $A\cap L_1\times L_2\subseteq U_2$ and $B\cap L_1\times L_2\subseteq V_2$.
Next, by Lemma \ref{L2} fix disjoint open sets $W_1$, and $W_2$ such that $L_1\times \omega_1\subseteq W_1$ and $\omega_1\times L_2\subseteq W_2$. 

Since ladder system spaces $X_{L}$ are normal under $MA(\omega_1)$ so are the subspaces $X_{L_1}\times \omega_1$ and $\omega_1\times X_{L_2}$ being free sums
 of $\omega_1$ copies of the normal subspaces.  So we may find disjoint open sets $U_1^1\subseteq W_1$ and $V_1^1\subseteq W_1$ such that 
$A\cap X_{L_1}\times \omega_1\subseteq U_1^1$ and $B\cap X_{L_1}\times \omega_1\subseteq V_1^1$. By Lemma \ref{L3} we can shrink $V_2$ and $U_2$ and also assume that $U_1^1 \cap V_2=V_1^1\cap U_2=\emptyset$. By the same argument, we can also find $U_1^2$ and $V_1^2$ separating $A\cap \omega_1\times X_{L_2}\ \omega_1$ and 
$B\cap \omega_1\cap X_{L_2}$ and may also assume that $U_1^2 \cap V_2=V_1^2\cap U_2=\emptyset$. So then we have that 
$U=U_2\cup U_1^1\cup U_1^2 (A\cap \omega_1\times \omega_1)$ is an open set containing $A$ 
and $V=V_2\cup V_1^1\cup V_1^2 (B\cap \omega_1\times \omega_1)$ is an open set containing $B$ and $U\cap V=\emptyset$.
 This completes the proof that $X_{L_1}\times X_{L_2}$ is hereditarily normal. 
\end{proof}

In our last problem concerning ladder system spaces, we consider ladder systems over a set of ordinals $D\subseteq \kappa$ where $\kappa>\cont$. It was shown in \cite{BE+} that it is consistent with CH that all ladder systems $L$ on any  subset of $\omega_1$ have $X_L$ countably paracompact, and so it is consistent that all ladder systems on $\cont$ determine $\Delta$-spaces. However, we don't know whether assuming only ZFC one can construct a ladder system on some $\kappa$ giving a ladder system space that is not in the class $\Delta$. 
\begin{problem} \label{problem7}
Does there exist in ZFC a ladder system $L$ on some $\kappa$ whose corresponding space $X_L$ 
is not countably metacompact, hence $X_L \notin \Delta$?
\end{problem}
In other words, we ask for a ZFC example of a ladder system $L$ on some uncountable cardinal that fails to have a relatively weak uniformization property (i.e. the property $CM(L)$ as formulated in \cite{BE+}). 
%%%%%%%%%%%%%%%%%%%%%%
\section{Some examples of $X \notin \Delta$}\label{Section1}
%%%%%%%%%%%%%%%%%%%%%
%Let $\cont$ denote the cardinality of continuum. 
It follows from Theorem \ref{Theor:height_2} that separable compact spaces $X \in \Delta$ with $|X|=\cont$
do exist in ZFC. On the other hand, no $\Delta$-set of reals can have cardinality $\cont$.  % \cite{P1}.
  Below we will extend the last claim for much more general classes of topological spaces.
Denote by $o(X)$ the cardinality of the family of all open sets in $X$.

\begin{theorem}\label{Theor:o(X)}
If $o(X)^{\aleph_0} \leq |X|$, then $X \notin \Delta$.
\end{theorem}
\begin{proof}
Our proof is based on the argument which appears (implicitly) in \cite{P1}.
Denote the cardinal $o(X)^{\aleph_0}$ by $\lambda$. Enumerate $X =\{x_{\alpha}: \alpha < \tau\}$.
On the contrary, assume that $X \in \Delta$.
Enumerate by $\{\{U_n^{\alpha}: n\in\omega\}:\alpha < \lambda\}$ all countable
sequences of open subsets of $X$ with empty intersection. By assumption, $\lambda \leq \tau$.
For every $\alpha < \lambda$ choose an $n(\alpha) \in \omega$ such that $x_{\alpha} \notin U_{n(\alpha)}^{\alpha}$.
Define $A_n = \{x_\alpha: n(\alpha) \geq n \}$. Clearly, $\bigcap_{n\in\omega} A_n =\emptyset$.
If there existed an $\alpha < \lambda$ such that $A_n \subset U_n^{\alpha}$, for each $n\in\omega$, then we would have
$x_\alpha \in A_{n(\alpha)} \subseteq U_{n(\alpha)}^{\alpha}$, which is a contradiction. This means that $X \notin \Delta$.
\end{proof}
The next Proposition \ref{cor:hereditarily} strengthens notably several results obtained in \cite{FKLS}.

\begin{proposition}\label{cor:hereditarily} {\rm (a)} Let $X$ be a hereditarily separable space. If $|X| =\cont$, then  $X \notin \Delta$.
\newline
{\rm (b)} Let $X$ be a separable hereditarily Lindel\" of space. If $|X| =\cont$, then  $X \notin \Delta$.
\end{proposition}
\begin{proof}
(a) For any $X$, the inequality $o(X) \leq |X|^{hd(X)}$ holds \cite{Hodel}.
 Since $hd(X)=\aleph_0$ we get that $o(X)^{\aleph_0} \leq \cont^{\aleph_0 \times \aleph_0} = |X|$, and Theorem \ref{Theor:o(X)} applies.
\newline
(b) For any (regular) $X$, the inequalities $w(X) \leq 2^{d(X)}$ and $o(X) \leq w(X)^{hL(X)}$ hold \cite{Hodel},
so again $o(X)^{\aleph_0} \leq \cont^{\aleph_0 \times \aleph_0} = |X|$, and Theorem \ref{Theor:o(X)} applies.
\end{proof}

Let $\S$ denote the Sorgenfrey line. It has been shown in \cite{FKLS} that $\S \notin \Delta$. Note that the proof presented in \cite{FKLS}
is valid for any subset of $\S$ containing a segment but it does not work for more complicated subspaces of $\S$. Proposition \ref{cor:hereditarily}(a)
immediately implies 

\begin{corollary}\label{cor:Sorgenfrey}
Let $X$ be any subspace of the Sorgenfrey line $\S$ with $|X|= \cont$. Then $X \notin \Delta$.
\end{corollary}

%\begin{remark}\label{remark:sep} By Proposition \ref{cor:hereditarily}(a), CH implies that there are no
%uncountable hereditarily separable $\Delta$-spaces. On the other hand, it is consistent that
%every subset $X \subset \R$ of cardinality less than $\cont$ is a $Q$-set.
%\end{remark}

\begin{remark}\label{remark:Lspace} 
It is not clear whether an assumption of separability can be omitted in Proposition \ref{cor:hereditarily}(b).
Assume that $X$ is a hereditarily Lindel\" of $\Delta$-space with $|X| =\cont$.
Then $X$ can not be separable by Proposition \ref{cor:hereditarily}(b). So, $X$ would be an $L$-space which is also a $\Delta$-space.
A highly nontrivial example of an $L$-space in ZFC was constructed by J. T. Moore \cite{Moore}. Recently it was shown \cite{PM} that Moore's $L$-space
 is not a $Q$-set space and under the
assumption that all Aronszajn trees are special, Moore's $L$-space is not a $\Delta$-space, but we don't know whether his construction can 
give an $L$-space which is also a $\Delta$-space.
Every hereditarily Lindel\" of scattered space contains a countable dense set of isolated points, hence an $L$-space
can not be homeomorphically embedded into a compact $\Delta$-space. 

 As mentioned in Introduction, Z. Balogh \cite{Balogh} constructed a hereditarily paracompact, perfectly normal $Q$-set space $X$ such that $|X|= \cont$. 
Dennis Burke informed the authors that there are handwritten notes by Z. Balogh where he started to outline a strategy for constructing
a Lindel\" of $Q$-set space of cardinality continuum. Such a space would be evidently a hereditarily Lindel\" of $\Delta$-space.
However, according to Dennis Burke, these notes are incomplete and seem to leave things hanging.
So, the following apparently very challenging problem remains open.  
\end{remark}

\begin{problem} \label{problem1}
Does there exist in ZFC an uncountable hereditarily Lindel\" of $\Delta$-space? or even an uncountable hereditarily Lindel\" of $Q$-set space?
\end{problem}

Recall that a space is {\it resolvable} ({\it $\omega$-resolvable})  if it can be partitioned into $2$ (countably many) dense subsets.

\begin{proposition} If $X$ is Baire and $\omega$-resolvable, then it is not a $\Delta$-space. 
\end{proposition}

\begin{proof} Fix a partition of $X$ into countably many dense sets $D_n$. If $U_n\supseteq D_n$ then $U_n$ is dense open and so $\bigcap_{n\in\omega} U_n\not=\emptyset$ and so $\{D_n:n\in \omega\}$ has no point-finite open expansion. 
\end{proof}
Since it is well known that a Souslin line is Baire and $\omega$-resolvable we obtain
\begin{corollary} A Souslin line is not a $\Delta$-space.
\end{corollary}

More generally, any Lindel\"of regular space with all open sets uncountable is resolvable \cite{Fil}.
This result was improved in \cite{JSS} to $\omega$-resolvable, and since Baire spaces without isolated points have all open subsets uncountable, we obtain

\begin{corollary}\label{Cor_Lind} A Lindel\"of Baire space without isolated points is not a $\Delta$-space. 
\end{corollary}

It follows, for example, that the classical $L$-spaces constructed from CH (see e.g. \cite{Roit}) are not $\Delta$-spaces.

Note that Corollary \ref{Cor_Lind} provides a strengthening of a result mentioned in the Introduction: compact $\Delta$-spaces must be scattered \cite{KL}.  
Whether there are more interesting examples of Baire spaces that are either $\Delta$-spaces or $Q$-set spaces is open, so we ask

\begin{problem}\label{prob_Baire} Does every Baire $\Delta$-space have an isolated point? Are there uncountable Baire $Q$-set spaces?
\end{problem}

Corollary \ref{Cor_Lind} also suggests a measure-theoretic analogue:
\begin{problem} If $X$ admits a strictly positive $\sigma$-additive measure vanishing on points is then $X\not\in \Delta$?
\end{problem}

As mentioned in Introduction, $\omega_1$ is an example of a first countable, locally compact, scattered space that is not a $\Delta$-space. 
We now characterize those subspaces of $\omega_1$ which are $\Delta$-spaces (a fact that will be useful for what follows). 

\begin{theorem}\label{Theor:omega1} For a subset $X\subseteq \omega_1$, the following conditions are equivalent{\rm:}
\begin{enumerate}
\item[{\rm (1)}] $X$ is a nonstationary set. 
\item[{\rm (2)}] $X$ is a $Q$-space.
\item[{\rm (3)}] $X$ is a $\Delta$-space.
\end{enumerate}
\end{theorem}
\begin{proof} If $X \subset \omega_1$  is nonstationary, then since $\omega_1 \setminus X$ contains a closed unbounded set,
$X$ is the free topological union of pairwise disjoint countable open sets.
It is easily seen that any countable set is a $Q$-space and any free union of $Q$-spaces is a $Q$-space. This proves (1) $\Rightarrow$ (2).
(2) $\Rightarrow$ (3) has been mentioned before.

It remains to prove (3) $\Rightarrow$ (1). Assume $X$ is stationary.
Let $L=\{\alpha \in \omega_1: \alpha\,\,\mbox{is a limit point of}\,\,X\}$. Then $X^{'} = X \cap L$ is also a stationary set.
Represent $X^{'}$ as a union of countably many pairwise disjoint stationary sets $\{X_n: n\in \omega\}$.

As an easy application of Fodor's Lemma the family $\{X_n: n\in \omega\}$ has no point-finite open expansion. 
 Indeed, if $Y\subseteq \omega_1$ is stationary, and $Y\subseteq U$, where $U$ is open, then $(\beta,\omega_1)\subseteq U$ for some $\beta$.
Therefore, if $U_n$ are open and $X_n\subseteq U_n$ for each $n\in\omega$, then there is an ordinal  $\gamma\in \omega_1$ such that 
$X\setminus \gamma\subseteq \bigcap_{n\in\omega} U_n$. It follows that the disjoint family $\{X_n:n\in \omega\}$ has no point-finite open expansion and so $X\not\in \Delta$.
\end{proof}

%Recall that a topological space $X$ is a {\it $P$-space }if countable intersections of open sets in $X$ are open. 
Now we give an example of a Lindel\"of $P$-space answering negatively the following question posed in
\cite[Problem~24]{FS}: if $X$ is a $P$-space, must $C_p(X)$ be distinguished?

\begin{example}\label{P-space}
The $G_{\delta}$-modification $W_{\delta}$ of a topological space $W$ is the space on the same
underlying set generated by the family of all $G_{\delta}$-sets of $W$.
It is known that if a compact space $W$ is scattered then $W_{\delta}$ is Lindel\" of \cite{Meyer}.
For instance, if $W$ is the compact scattered ordinal space $[0, \omega_1]$, then $X = W_{\delta}$ is the one-point
Lindel\" ofication of a discrete set. 

Let $W$ be the compact scattered ordinal space $[0, \omega_2]$. Then $X = W_{\delta}$ is a Lindel\" of $P$-space
such that each ordinal $\alpha$ with countable cofinality is isolated in $X$,
and each ordinal $\alpha$ with uncountable cofinality is a limit point in $X$.
Note that the set of all ordinals in $\omega_2$ with cofinality  $\omega_1$ is stationary.
Therefore, repeating the argument of above Theorem \ref{Theor:omega1}, again by the Fodor's lemma,
we show that $ W_{\delta}\notin \Delta$.
\end{example}

\begin{problem} \label{problem2}
Find a characterization of those scattered compact spaces $W$ such that $W_{\delta} \in \Delta$.
\end{problem}

We end the paper with some observations about the relationship between $\Delta$-spaces and $Q$-set spaces.
Recall that Balogh's definition of a $Q$-set space requires the space to be not $\sigma$-discrete. He was interested in the existence of non-trivial examples of spaces where every subset was $G_\delta$. 
% since any $\sigma$-discrete $Q$-space is $\sigma$-closed discrete. Any $\sigma$-closed discrete space trivially is a $Q$-space.
%Z.\ Balogh defined a $Q$-set space to be a space $X$ so that every subset of $X$ is a $G_\delta$ but $X$ itself is not $\sigma$-discrete. He gave ZFC constructions of $Q$-set spaces, one even that is paracompact (\cite {Balogh}). 
We have described in the paper a number of compact $\Delta$-spaces in ZFC that contain subsets which are not $G_\delta$ (e.g. Examples \ref{Theor:height_2} and \ref{example2}),
 but all of them are $\sigma$-discrete and have at least one point that is not a $G_\delta$. In \cite{Casas} J.\ Casas de la Rosa observed that
 the Alexandroff duplicate $X$ of a $Q$-set space is a $\Delta$-space which is not $\sigma$-discrete and all singletons $\{x\}$ are $G_\delta$-sets.
 But this space $X$ does include a closed non-$G_\delta$ subset. So we ask 
\begin{problem} Is there in ZFC a $\Delta$-space which is not $\sigma$-discrete, has all closed sets $G_\delta$ but is not a $Q$-set space?
\end{problem}
Perhaps a modification of the Balogh-Rudin technique that Balogh used to construct ZFC examples of $Q$-set spaces and Dowker spaces could be used to give a positive answer. 

S. Shelah \cite{Shelah} showed that consistently there is a normal ladder system space $X_L$ (hence $X_L$ is a $\Delta$-space)
such that the closed discrete set of non-isolated points is not $G_{\delta}$.
On the other hand, under PMEA, in a first-countable, countably metacompact $T_1$ space every closed discrete subset is a $G_{\delta}$-set \cite{Burke}. 
Moreover, as we have mentioned in Remark \ref{remark:DS}, $MA(\omega_1)$ implies that every subset of a ladder system space $X_L$ is $G_{\delta}$.

For a subclass of $\Delta$-spaces which includes all Tychonoff spaces embeddable into a scattered Eberlein compact space,
we have a positive result, in ZFC.
A family $\left\{\mathcal{N}_{x}:x\in X\right\}$ of subsets of a topological space $X$ is called a {\it point-finite neighborhood assignment} 
 if each $\mathcal{N}_{x}$ is an open neighbourhood of $x$ and for each $u\in X$ the set $\left\{ x\in X:u\in \mathcal{N}_{x}\right\}$ is finite.
It is proved in \cite[Theorem 46]{FKLS} that if $X$ admits a point-finite neighborhood assignment $\left\{ \mathcal{N}_{x}:x\in X\right\}$ then $X\in\Delta$.
Also, a compact space $X$ admits a point-finite neighborhood assignment if and only if $X$ is a scattered Eberlein compact space \cite[Theorem 51]{FKLS}.

\begin{theorem} \label{Th:pfna} Assume that a topological space $X$ admits a point-finite neighborhood assignment. 
If $F$ is any subset of $X$ consisting of points which are $G_{\delta}$ in $X$,
then $F$ is a $G_{\delta}$-set in $X$.
\end{theorem}
\begin{proof} Let $\left\{\mathcal{N}_{x}:x\in X\right\}$ be a point-finite neighborhood assignment in $X$.
For every point $x\in F$ fix the sequence of open sets $\{U_n(x): n\in \omega\}$ such that 
$$U_0(x) \subseteq \mathcal{N}_{x}, U_{n+1}(x) \subseteq U_n(x)\,\, \text{and}\,\, \bigcap\{U_n(x): n\in\omega\} = \{x\}.$$
Define open sets $V_n = \bigcup_{x\in F} U_n(x)$ for every $n\in\omega$.
We claim that $\bigcap_{n\in\omega} V_n = F$. Indeed, let $y \notin F$ be any point. There are at most finitely many points $x_1, x_2, \dots, x_k$ in $F$ such that
$y \in \mathcal{N}_{x_i}$. Since $\bigcap_{n\in\omega} U_n(x_i) = \{x_i\}$ for each $i = 1, 2, \dots, k$, there is one index $n_0$ large enough such that 
$y \notin U_n(x_i)$ for every $n > n_0$. Finally, $y \notin V_n$ for every $n > n_0$.
\end{proof}

Since as we have mentioned before, every scattered Eberlein compact space admits a point-finite neighborhood assignment we immediately derive
\begin{corollary}\label{cor:Qspace} Let $F$ be any subspace of a scattered Eberlein compact space $X$ consisting of points which are $G_{\delta}$ in $X$.
Then $F$ is a $Q$-space.
\end{corollary}

Note that Theorem \ref{Th:pfna} generalizes \cite[Theorem 1]{KMM}.

\textbf{Acknowledgements.} 
The authors are grateful to the referee for careful reading of the paper and valuable 
suggestions and comments.


\begin{thebibliography}{99}

\bibitem{Balogh1} Zolt\' an T. Balogh,
\newblock \textit{There is a  $Q$-set space in ZFC},
\newblock Proc. Amer. Math. Soc. \textbf{113} (1991), 557-561.

\bibitem{Balogh_Dowker} Zolt\' an T. Balogh,
\newblock \textit{A small Dowker space in ZFC},
\newblock Proc. Amer. Math. Soc. \textbf{124} (1996), 2555--2560.
 
\bibitem{Balogh} Zolt\' an T. Balogh,
\newblock \textit{There is a paracompact $Q$-set space in ZFC},
\newblock Proc. Amer. Math. Soc. \textbf{126} (1998), 1827--1833. 

\bibitem{BE+} Zolt\' an Balogh, Todd Eisworth, Gary Gruenhage, Oleg Pavlov and Paul Szeptycki,
\newblock \textit{Uniformization and anti-uniformization properties of ladder systems},
\newblock Fund. Math. \textbf{181} (2004), 189--213.

\bibitem{Burke} Dennis K. Burke,
\newblock \textit{Closed discrete sets in first countable, countably metacompact spaces},
\newblock Topology Appl. \textbf{44} (1992), 63--67.

\bibitem{Casas} Javier Casas de la Rosa, \textit{Private communication} 2020. 

\bibitem{Devlin} Keith J. Devlin,
\newblock \textit{Note on a theorem of J. Baumgartner}, 
\newblock Fund. Math. \textbf{76} (1972), 255--260.

\bibitem{DS} Keith Devlin and Saharon Shelah,
\newblock \textit{A weak version of $\bDiamond$ which follows from $2^{\aleph_0} < 2^{\aleph_1}$}, 
\newblock Israel J. Math. \textbf{29} (1978), 239--247.

\bibitem{vanDouwen} Eric K. van Douwen,
\newblock \textit{The integers and topology},
\newblock in: Handbook of Set-Theoretic Topology, Editors: K. Kunen, J. E. Vaughan,  North-Holland, Amsterdam, 1984, 111--167.

\bibitem{Dow} Alan Dow and Jerry E. Vaughan, 
{\it Ordinal remainders of classical $\psi$-spaces},
Fund. Math. \textbf{217} (2012), 83--93.

\bibitem{Engelking} Richard Engelking,
\newblock \textit{General Topology},
\newblock Heldermann Verlag, Berlin, 1989.

\bibitem{FS} J. C. Ferrando and Stephen S. A. Saxon,
\newblock \textit{If not distinguished, is $C_p(X)$ even closed?},
\newblock Proc. Amer. Math. Soc. \textbf{149} (2021), 2583--2596.

\bibitem{FKLS} J. C. Ferrando, J. K\c akol, A. Leiderman, S. A. Saxon,
\newblock \textit{Distinguished $C_{p}\left( X\right)$ spaces},
\newblock Rev. R. Acad. Cienc. Exactas Fis. Nat. Ser. A Mat. RACSAM \textbf{115} (2021), issue 1,
\newblock https://doi.org/10.1007/s13398-020-00967-4

\bibitem{Fil} M. A.\ Filatova, \newblock \textit{Resolvability of Lindel\" of spaces},
\newblock J. Math. Sci. (N.Y.) \textbf{146}, No. 1, 2007, 5603--5607.

\bibitem{Fleissner} William G. Fleissner,
\newblock \textit{Remarks on Souslin properties and tree topologies}, 
\newblock Proc. Amer. Math. Soc. \textbf{80} (1980), 320--326.

\bibitem{FM} William G. Fleissner and Arnold W. Miller,
\newblock \textit{On $Q$-sets}, 
\newblock Proc. Amer. Math. Soc. \textbf{78} (1980), 280--284.

\bibitem{Fremlin} D. H. Fremlin,
\newblock \textit{Consequences of Martin's axiom},
\newblock Cambridge University Press, Cambridge, 1984.

\bibitem{Hanazawa} Masazumi Hanazawa,
\newblock \textit{Countable metacompactness and tree topologies}, 
\newblock J. Math. Soc. Japan \textbf{35} (1983), 59--70.

\bibitem{Hart} Klaas Pieter Hart,
\newblock \textit{Characterizations of $\R$-embeddable and developable $\omega_1$-trees},
\newblock Indag. Math. \textbf{85} (1982), 277--283.

\bibitem{Hausdorff} F. Hausdorff,
\newblock \textit{Probl$\grave{e}$me 58},
\newblock Fund. Math. \textbf{20} (1933), p. 286. 

\bibitem{Hodel} R. Hodel,
\newblock \textit{Cardinal functions I},
\newblock in: Handbook of Set-Theoretic Topology, Editors: K. Kunen, J. E. Vaughan,  North-Holland, Amsterdam, 1984, 1--61.

\bibitem{HM} Michael Hru\v s\'ak and Justin Tatch Moore,
\newblock \textit{Introduction: Twenty problems in set-theoretic topology},
 in: Open Problems in Topology II, E. Pearl, ed., Elsevier, Amsterdam, 2007, 111--114.

\bibitem{JSS} Istv\' an Juh\' asz, Lajos Soukup and Zolt\' an Szentmikl\' ossy,
\newblock \textit{Regular spaces of small extent are $\omega$-resolvable},
 Fund. Math. \textbf{228} (2015), 27--46.

\bibitem{Junnila} Heikki Junnila,
\newblock \textit{Some topological consequences of the Product Measure Extension Axiom},
\newblock Fund. Math \textbf{115} (1983), 1--8.

\bibitem{KL} Jerzy K\c akol and Arkady Leiderman,
\newblock \textit{A characterization of $X$ for which spaces $C_p(X)$ are distinguished and its applications},
\newblock Proc. Amer. Math. Soc., series B,  \textbf{8} (2021), 86--99.

\bibitem{KL2} Jerzy K\c akol and Arkady Leiderman,
\newblock \textit{Basic properties of $X$ for which the space $C_p(X)$ is distinguished},
\newblock Proc. Amer. Math. Soc., series B,  \textbf{8} (2021), 267--280.

\bibitem{Knight} R. W. Knight,
\newblock \textit{$\Delta$-Sets},
\newblock Trans. Amer. Math. Soc. \textbf{339} (1993), 45--60.

\bibitem{KS} Menachem Kojman and Saharon Shelah,
\newblock \textit{A ZFC Dowker space in $\aleph_{\omega +1}$: an application of PCF theory to topology},
\newblock Proc. Amer. Math. Soc. \textbf{126} (1998), 2459--2465.

\bibitem{KMM} Adam Krawczyk, Witold Marciszewski, Henryk Michalewski, 
\newblock \textit{Remarks on the set of $G_{\delta}$-points in Eberlein and Corson compact spaces},
\newblock Topology Appl. \textbf{156} (2009), 1746--1748.

\bibitem{Kunen} Kenneth Kunen, 
\newblock \textit{Set Theory},
\newblock Studies in Logic (London), vol. 34, College Publications, London, 2011. 
 
\bibitem{Lei_Tkachenko}
Arkady Leiderman and Mikhail Tkachenko, 
{\it  Some properties of one-point extensions},
 Topology Proc. \textbf{59} (2022), 195--208.

\bibitem{Lei_Tkachuk}
A. Leiderman and V. V. Tkachuk,
\newblock \textit{Pseudocompact $\Delta$-spaces are often scattered},
\newblock Monatsh. Math. \textbf{197} (2022), 493--503.

\bibitem{PM} P. Memarpanahi, {\it Moore's L-space, $Q$-set spaces and $\Delta$-set spaces}, (in preparation).

\bibitem{Meyer} P. R. Meyer,
\newblock \textit{The Baire order problem for compact spaces},
\newblock Duke Math. J. \textbf{33} (1966), 33--40.

\bibitem{Miller} Arnold W. Miller, 
\newblock \textit{Special subsets of real line},
\newblock in: The Handbook of Set-Theoretic Topology, Editors: K. Kunen, J. E. Vaughan, North-Holland, Amsterdam, 1984, 201--233.

\bibitem{Moore} Justin Tatch Moore,
\newblock \textit{A solution to the $L$ space problem},
\newblock J. Amer. Math. Soc. \textbf{19} (2005), 717--736.

\bibitem{Nyikos} Peter J. Nyikos,
\newblock \textit{Various topologies on trees},
\newblock in: Proceedings of the Tennessee Topology Conference, Editors: P. R. Misra and M. Rajagopalan,
 World Scientific Publishing Co., 1997, 167--198.

\bibitem{P1} T. C. Przymusi\'{n}ski,
\newblock \textit{Normality and separability of Moore spaces},
\newblock in: Set-Theoretic Topology, Academic Press, New York, 1977, 325--337.

\bibitem{Reed} G. M. Reed,
\newblock \textit{On normality and countable paracompactness},
\newblock Fund. Math. \textbf{110} (1980), 145--152.

\bibitem{Reed2} G. M. Reed,
\newblock \textit{Set-theoretic problems in Moore spaces},
\newblock in: Open Problems in Topology, Editors: J. van Mill, G. M. Reed, North-Holland, Amsterdam, 1990, 163--181.

\bibitem{Roit} Judy Roitman,
\newblock \textit{Basic $S$ and $L$},
\newblock in: Handbook of Set-Theoretic Topology, Editors: K. Kunen, J. E. Vaughan, North-Holland, Amsterdam, 1984, 295--326.

\bibitem{Rothberger} F. Rothberger,
\newblock \textit{On some problems of Hausdorff and of Sierpi\'nski},
\newblock Fund. Math. \textbf{35} (1948), 29--46.

\bibitem{Rudin} M. E. Rudin,
\newblock \textit{A normal space $X$ for which $X\times I$ is not normal},
\newblock Fund. Math. \textbf{73} (1971), 179--186.

\bibitem{Shelah} S. Shelah,
\newblock \textit{A consistent counterexample in the theory of collectionwise Hausdorff spaces},
\newblock Israel J. Math. \textbf{65} (1989), 219--224.

\bibitem{Sierpinski} W. Sierpi\'nski,
\newblock \textit{Sur un probl$\grave{e}$me de M. Hausdorff},
\newblock Fund. Math. \textbf{30} (1938), 1--7.

\bibitem{Paul_SZ} Paul J. Szeptycki,
\newblock \textit{Countable metacompactness in  $\Psi$-spaces}, 
\newblock Proc. Amer. Math. Soc. \textbf{120} (1994), 1241--1246.

\bibitem{Stevo} S. Todor\v cevi\' c,
\newblock \textit{Trees and linearly ordered sets},
\newblock in: Handbook of Set-Theoretic Topology, Editors: K. Kunen, J. E. Vaughan,  North-Holland, Amsterdam, 1984, 235--293. 

\bibitem{Stevo1} Stevo Todor\v cevi\' c,
\newblock \textit{Real functions on the family of all well-ordered subsets of a partially ordered set}, 
\newblock J. Symb. Logic. \textbf{48} (1983), 91--96.

\bibitem{Stevo2} Stevo Todor\v cevi\'c, \newblock \textit{Representing trees as relatively compact subsets of the first Baire class},
\newblock Bulletin T. CXXXI de l’Acad\' emie serbe des sciences et des arts,
Classe des Sciences math\' ematiques et naturelles,
Sciences math\' ematiques, \textbf{30} (2005), 29--45.

\end{thebibliography}
\end{document}